\documentclass[11pt]{amsart}

\usepackage{a4wide}
\usepackage{amssymb}
\usepackage{float}
\usepackage{graphicx}
\usepackage{mathrsfs}
\usepackage{enumerate}
\usepackage{esint}
\usepackage{titletoc}
\usepackage[colorlinks=true, urlcolor=blue, linkcolor=blue, citecolor=black]{hyperref}
\usepackage[nameinlink]{cleveref}

\theoremstyle{plain}
\newtheorem{theorem}{Theorem}[section]
\newtheorem{lemma}[theorem]{Lemma}

\theoremstyle{definition}
\newtheorem{definition}[theorem]{Definition}
\newtheorem{remark}[theorem]{Remark}

\numberwithin{equation}{section}

\renewcommand{\d}{\textnormal{d}}
\newcommand{\N}{\mathbb{N}}
\newcommand{\R}{\mathbb{R}}

\DeclareMathOperator{\supp}{supp}

\makeatletter
\@namedef{subjclassname@2020}{%
  \textup{2020} Mathematics Subject Classification}
\makeatother

\title[]{Local regularity for nonlocal equations with variable exponents}

\author{Jamil Chaker}
\address{Fakult\"at f\"ur Mathematik, Universit\"at Bielefeld, 33615 Bielefeld, Germany}
\email{jchaker@math.uni-bielefeld.de}

\author{Minhyun Kim}
\address{Fakult\"at f\"ur Mathematik, Universit\"at Bielefeld, 33615 Bielefeld, Germany}
\email{minhyun.kim@uni-bielefeld.de}

\thanks{Jamil Chaker gratefully acknowledges funding by the Deutsche Forschungsgemeinschaft (SFB 1283/2 2021 – 317210226). Minhyun Kim gratefully acknowledges funding by the Deutsche Forschungsgemeinschaft (GRK 2235/2 2021 - 282638148).}

\begin{document}

\begin{abstract}
In this paper, we study local regularity properties of minimizers of nonlocal variational functionals with variable exponents and weak solutions to the corresponding Euler--Lagrange equations. We show that weak solutions are locally bounded when the variable exponent $p$ is only assumed to be continuous and bounded. Furthermore, we prove that bounded weak solutions are locally H\"older continuous under some additional assumptions on $p$. On the one hand, the class of admissible exponents is assumed to satisfy a log-H\"older-type condition inside the domain, which is essential even in the case of local equations. On the other hand, since we are concerned with nonlocal problems, we need an additional assumption on $p$ outside the domain.
\end{abstract}

\maketitle


\section{Introduction} \label{sec:introduction}

The aim of this paper is to study the regularity theory for minimizers of the nonlocal variational functional
\begin{equation} \label{eq:F}
\mathcal{F}(u) = \int_{\mathbb{R}^n}\int_{\mathbb{R}^n} \frac{|u(x)-u(y)|^{p(x,y)}}{p(x,y) |x-y|^{n+sp(x,y)}} \,\mathrm{d}y\,\mathrm{d}x
\end{equation}
and for weak solutions to the corresponding Euler--Lagrange equation, where $n \in \mathbb{N}$, $s \in (0,1)$, and $p:\mathbb{R}^n \times \mathbb{R}^n \to \mathbb{R}$ is a continuous function such that
\begin{equation} \label{eq:symmetry}
p(x,y) = p(y,x)
\end{equation}
and
\begin{equation} \label{eq:bounds}
1 < \inf_{x, y \in \mathbb{R}^n} p(x,y) \leq \sup_{x, y \in \mathbb{R}^n} p(x,y) < +\infty.
\end{equation}
This functional is a nonlocal analog of a local variational functional
\begin{equation} \label{eq:F_loc}
\mathcal{F}_{\mathrm{loc}}(u) = \int_{\Omega} \frac{1}{p(x)} |Du(x)|^{p(x)} \,\mathrm{d}x,
\end{equation}
where $\Omega$ is a bounded domain in $\mathbb{R}^n$ and $p: \Omega \to \mathbb{R}$ is a measurable function such that $1 < \inf_{x \in \Omega} p(x) \leq \sup_{x \in \Omega} p(x) < +\infty$. The functional in \eqref{eq:F_loc} was first considered by Zhikov \cite{Zhi86,Zhi93}. The regularity properties for minimizers of \eqref{eq:F_loc} or more general local variational functionals have been established in several works. See for instance \cite{Mar89,Mar91,AF94,Zhi95,Alk97,CPC97,Zhi97,CM99,FZ99,AM01,HKLMP08,HHLN10,Bre12,Ok18,DZZ20} and the references therein.

A function $u \in W^{s,p(\cdot,\cdot)}(\mathbb{R}^n)$ is said to be a {\it minimizer of $\mathcal{F}$ in $\Omega$} if
\begin{equation*}
\mathcal{F}(u) \leq \mathcal{F}(u+\varphi)
\end{equation*}
for any measurable function $\varphi:\mathbb{R}^n \to \mathbb{R}$ supported inside $\Omega$. See \Cref{sec:Sobolev} for the definition of the function space $W^{s,p(\cdot,\cdot)}(\mathbb{R}^n)$. It is standard to show that minimizers of $\mathcal{F}$ in $\Omega$ are weak solutions to the Euler--Lagrange equation
\begin{equation} \label{eq:EL}
(-\Delta)^s_{p(\cdot,\cdot)}u = 0
\end{equation}
in $\Omega$, where $(-\Delta)_{p(\cdot,\cdot)}^{s}$ is the fractional $p(\cdot,\cdot)$-Laplacian defined by
\begin{equation*}
(-\Delta)_{p(\cdot,\cdot)}^{s}u(x) = \mathrm{P.V.} \int_{\mathbb{R}^n} \frac{|u(x)-u(y)|^{p(x,y)-2}(u(x)-u(y))}{|x-y|^{n+sp(x,y)}} \,\mathrm{d}y, \quad x \in \mathbb{R}^n.
\end{equation*}
See \Cref{sec:Caccioppoli} for the precise definition of weak solution.

Before we formulate the assumptions on $p$ and the main results of this paper, let us recall the regularity results for local variational functionals and the corresponding local operators. It is known \cite{FZ99} that minimizers of \eqref{eq:F_loc} in $\Omega$ and weak solutions to the corresponding Euler--Lagrange equation $-\Delta_{p(\cdot)}u=0$ in $\Omega$ are locally bounded in $\Omega$, provided that $p: \Omega \to \mathbb{R}$ is continuous on $\Omega$. Moreover, if the modulus of continuity $\omega$ of $p$ satisfies 
\begin{equation} \label{eq:log-Holder}
\limsup_{R \to 0} \omega(R) \log \left( \frac{1}{R} \right) < +\infty,
\end{equation}
then the minimizers and weak solutions are locally H\"older continuous. The log-H\"older continuity \eqref{eq:log-Holder} is sharp in the sense that regularity properties such as H\"older continuity and even higher integrability fail to hold if the condition \eqref{eq:log-Holder} is violated (see \cite{Zhi97}). Moreover, it is proved \cite{Zhi97} that the functional \eqref{eq:F_loc} exhibits the Lavrentiev phenomenon if and only if the condition \eqref{eq:log-Holder} is dropped. Furthermore, the singular part of the measure representation of relaxed integrals with variable exponent disappears if and only if \eqref{eq:log-Holder} holds (see \cite{ABF03}).

The log-H\"older-type condition \eqref{eq:log-Holder} is equivalent to the condition that there exists a constant $L > 0$ such that
\begin{equation*}
R^{p_-(B_R(x_0)) - p_+(B_R(x_0))} \leq L \quad \text{for all}~ \overline{B_R(x_0)} \subset \Omega,
\end{equation*}
where
\begin{equation*}
p_+(E) := \sup_{x \in E} p(x) \quad\text{and}\quad p_-(E) := \inf_{x \in E} p(x),
\end{equation*}
see \cite{Die04}. It is natural to expect that a similar condition on $p$ is required to obtain H\"older regularity results for the nonlocal variational functional \eqref{eq:F} and the nonlocal equation \eqref{eq:EL}. We introduce the following condition on $p$.

\begin{definition}
We say that a function $p: \mathbb{R}^n \times \mathbb{R}^n \to \mathbb{R}$ satisfies the condition \eqref{eq:P1} in $\Omega$ if there exists a constant $L > 0$ such that
\begin{equation} \label{eq:P1} \tag{P1}
R^{p_-(B \times B) - p_+(B \times B)} \leq L \quad \text{for all}~ \overline{B} = \overline{B_R(x_0)} \subset \Omega,
\end{equation}
where
\begin{equation*}
p_+(E \times F) = \sup_{x \in E, y \in F} p(x,y) \quad\text{and}\quad p_-(E \times F) = \inf_{x \in E, y \in F} p(x,y).
\end{equation*}
\end{definition}

Since we are concerned with nonlocal problems, we also need the information of $p$ outside the domain.

\begin{definition}
We say that a function $p: \mathbb{R}^n \times \mathbb{R}^n \to \mathbb{R}$ satisfies the condition \eqref{eq:P2} in $\Omega$ if
\begin{equation} \label{eq:P2} \tag{P2}
p_+(B \times B^c) \leq p_+(B\times B) \quad\text{and}\quad p_-(B \times B^c) \leq p_-(B\times B) \quad\text{for all}~ \overline{B} = \overline{B_R(x_0)} \subset \Omega.
\end{equation}
\end{definition}

Let us make some comments on the conditions \eqref{eq:P1} and \eqref{eq:P2}.

\begin{remark}{\,  }
\begin{enumerate}[(i)]
\item
Note that the condition \eqref{eq:P1} does not imply that $p$ is log-H\"older continuous as a $2n$-variable function, since $B \times B$ in \eqref{eq:P1} is not a ball with respect to the Euclidean metric in $\mathbb{R}^{2n}$. The condition \eqref{eq:P1} is actually weaker than the log-H\"older continuity of $p$. Let us first prove that the log-H\"older continuity of $p$ implies \eqref{eq:P1}. If $p$ is log-H\"older continuous, that is,
\begin{equation*}
|p(x_1, y_1)-p(x_2, y_2)| \leq \frac{C}{-\log \sqrt{|x_1-x_2|^2 + |y_1-y_2|^2}}
\end{equation*}
for all $(x_1, y_1), (x_2, y_2) \in \Omega \times \Omega$ with $\sqrt{|x_1-x_2|^2 + |y_1-y_2|^2} \leq 1/2$, then
\begin{equation*}
|p_-(B \times B) - p_+(B \times B)| \leq \frac{C}{-\log (2\sqrt{2}R)}
\end{equation*}
for any $\overline{B} = \overline{B_R(x_0)} \subset \Omega$ with $R \leq 1/8$. Thus, 
\begin{equation*}
R^{p_-(B\times B) - p_+(B\times B)} \leq R^{C/\log(2\sqrt{2}R)} = \exp\left( C\frac{\log R}{\log(2\sqrt{2}R)} \right) \leq e^{2C}.
\end{equation*}
If $R > 1/8$, then $R^{p_-(B\times B) - p_+(B\times B)} \leq 8^{p_+(B\times B) - p_-(B\times B)} \leq 8^{\|p\|_\infty}$. Therefore, \eqref{eq:P1} is proved for any $R > 0$.

Let us next provide an example of $p$ that is not log-H\"older continuous, but satisfies the condition \eqref{eq:P1}. The example will be given in $\mathbb{R} \times \mathbb{R}$, but it can be easily extended to $\mathbb{R}^n \times \mathbb{R}^n$. Let $\omega$ be a modulus of continuity that is smooth, bounded, concave, increasing, and satisfies
\begin{equation} \label{eq:mc}
\lim_{R \to 0} \frac{1}{-\log R}\frac{1}{\omega(R)} = 0.
\end{equation}
Define $p(x,y) = |x| \omega(|y|)$, then $p$ is clearly not log-H\"older continuous by \eqref{eq:mc}. To show that $p$ satisfies \eqref{eq:P1} in $(-1,1)$, let $x, y \in B:= (x_0-R, x_0+R)$ with $R < 1$. Then,
\begin{equation*}
\begin{split}
&p_+:= p_+(B \times B) = (|x_0|+R) \omega(|x_0|+R) \quad\text{and} \\
&p_-:= p_-(B \times B) = 
\begin{cases}
0 &\text{if}~ |x_0| < R, \\
(x_0-R) \omega(x_0-R) &\text{if}~ x_0 \geq R, \\
(x_0+R) \omega(x_0+R) &\text{if}~ x_0 \leq - R.
\end{cases}
\end{split}
\end{equation*}
When $|x_0|<R$, we have
\begin{equation*}
R^{p_--p_+} = R^{-(|x_0|+R)\omega(|x_0|+R)} \leq R^{-2R\omega(2R)} \leq R^{-2\|\omega\|_\infty R} \leq 2\|\omega\|_{\infty}.
\end{equation*}
If $x_0 \geq R$, then by the mean value theorem,
\begin{equation*}
p_+ - p_- = 2R f'(x_\ast)
\end{equation*}
for some $x_\ast \in B$, where $f(t) = t \omega(t)$. Since $f$ is concave and bounded, we have $f'(t) = \omega(t) + t\omega'(t) \leq 2\omega(t) \leq 2\|\omega\|_{\infty}$. Thus, 
\begin{equation*}
R^{p_--p_+} \leq R^{-2Rf'(x_\ast)} \leq R^{-4\|\omega\|_{\infty} R} \leq 4\|\omega\|_{\infty}.
\end{equation*}
The case $x_0 \leq -R$ can be treated in the same way. Therefore, $p$ satisfies \eqref{eq:P1} in $(-1,1)$.

For an explicit example of such $p$, one can consider a modulus of continuity $\omega$ that behaves like $1/\log(-\log R)$ or $\log(\log(1/ R))/(-\log R)$ near zero.

\item
Let us provide a nontrivial example of a function $p: \mathbb{R}^n \times \mathbb{R}^n \to \mathbb{R}$ satisfying \eqref{eq:P1} and \eqref{eq:P2} in $B_1$. Let $\omega$ be given by
\begin{equation*}
\omega(r) = 
\begin{cases}
3-\left( -\dfrac{1}{\log r} \land 1 \right) &\text{if}~ r < \frac{1}{e}, \\
\omega_0(r) &\text{if}~ r \geq \frac{1}{e},
\end{cases}
\end{equation*}
where $\omega_0$ is any non-increasing function such that $\omega_0(1/e)=2$ and $\lim_{r\to\infty}\omega_0(r) > 1$, and define $p: \mathbb{R}^n \times \mathbb{R}^n \to \mathbb{R}$ by $p(x,y)= \omega(|x-y|)$ (see Figure 1). Then, $p$ satisfies \eqref{eq:P1} because $p$ is log-H\"older continuous as a $2n$-variable function. Moreover, for any $\overline{B}=\overline{B_R(x_0)} \subset B_1$, we have
\begin{equation*}
\begin{split}
p_+(B \times B^c) = 3 = p_+(B \times B) \quad\text{and}\quad p_-(B \times B^c) = 2 \leq \omega(2R) = p_-(B \times B).
\end{split}
\end{equation*}
Therefore, $p$ also satisfies \eqref{eq:P2}.

\begin{figure}[H]
\centering
\includegraphics[width=.5\linewidth]{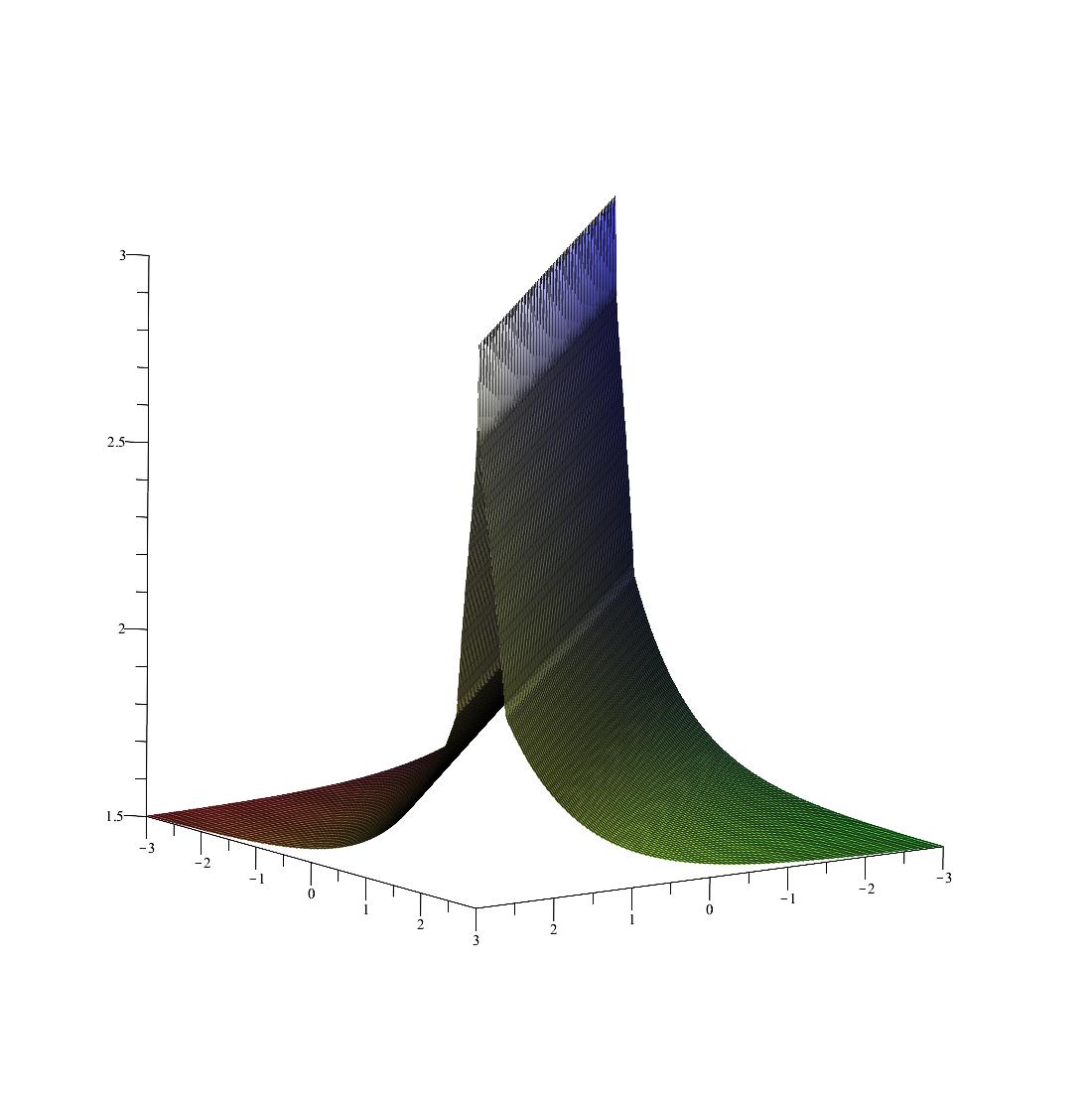}
\caption{Visualization of an example.}
\end{figure}

\item
The conditions \eqref{eq:P1} and \eqref{eq:P2} do not restrict $p$ on $\Omega^c \times \Omega^c$. In fact, for the local regularity results we do not need any information about $p$ on $\Omega^c \times \Omega^c$ except for the global bound \eqref{eq:bounds}. This is because the double integral over $\Omega^c \times \Omega^c$ vanishes whenever we use a cutoff function.
\end{enumerate}
\end{remark}

Let us now present the main results of this paper. The first result is the local boundedness of weak solutions to \eqref{eq:EL}. Throughout the paper, we always assume that $p: \mathbb{R}^n \times \mathbb{R}^n \to \mathbb{R}$ is a continuous function satisfying \eqref{eq:symmetry} and \eqref{eq:bounds}. Note that the following theorem does not require the conditions \eqref{eq:P1} and \eqref{eq:P2}.

\begin{theorem} \label{thm:local_boundedness}
Let $\Omega$ be a bounded domain in $\mathbb{R}^n$. If $u \in W^{s, p(\cdot,\cdot)}(\mathbb{R}^n)$ satisfies
\begin{equation} \label{eq:tail}
\sup_{x\in \Omega} \int_{\mathbb{R}^n} \frac{u_+(y)^{p(x,y)-1}}{1+|y|^{n+sp(x,y)}} \,\mathrm{d}y < +\infty \quad \left( \sup_{x\in \Omega} \int_{\mathbb{R}^n} \frac{u_-(y)^{p(x,y)-1}}{1+|y|^{n+sp(x,y)}} \,\mathrm{d}y < +\infty,~ \text{respectively} \right)
\end{equation}
and is a weak subsolution (supersolution, respectively) to \eqref{eq:EL} in $\Omega$, then $u$ is locally bounded from above (below, respectively) in $\Omega$. If $u \in W^{s, p(\cdot,\cdot)}(\mathbb{R}^n)$ is a weak solution to \eqref{eq:EL} in $\Omega$ satisfying \eqref{eq:tail} with $u_+$ replaced by $|u|$, then $u \in L^\infty_{\mathrm{loc}}(\Omega)$.
\end{theorem}

As a consequence of \Cref{thm:local_boundedness}, we know that every minimizer of \eqref{eq:F} in $\Omega$ is locally bounded in $\Omega$ since it is a weak solution to \eqref{eq:EL} in $\Omega$.

Indeed, we prove a stronger result than \Cref{thm:local_boundedness}, which provides a quantitative estimate of the supremum $u$, see \Cref{thm:local_boundedness2}. The strategy for the proof of \Cref{thm:local_boundedness} or \Cref{thm:local_boundedness2} is to develop the De Giorgi theory for the nonlocal functional $\mathcal{F}$ with variable exponent. This approach for nonlocal functionals with constant exponent has been studied extensively in the last few years. See for instance \cite{Kas09,CCV11,FK13,KS14,DK20} for the case $p=2$, and \cite{DCKP14,DCKP16,BP16,KMS15,Coz17} for $p > 1$. For a deeper discussion on fractional De Giorgi classes and their applications for the regularity of nonlocal problems, we refer the reader to \cite{Coz19} and the references therein. Analogously, we obtain the Caccioppoli-type estimate that contains terms with variable exponents, and then use the De Giorgi iteration technique to establish \Cref{thm:local_boundedness}. Due to the variable exponent in the Caccioppoli-type estimate, an additional difficulty arises in the De Giorgi iteration that does not occur in the case of the constant exponent. That is, different exponents involving $p_+$ and $p_-$ come into play in the iteration. Thus, the supremum of $u$ is controlled by a maximum of two $L^{p_+}$-norms of $u$ with different powers, and the nonlocal tail term having the variable exponent (see \Cref{thm:local_boundedness2}).

Let us mention that local boundedness of weak solutions to more general problems involving subcritical nonlinearity has been recently settled by Ho and Kim \cite{HK19}. However, their result requires an additional log-H\"older-type assumption on $p$ to cover the subcritical nonlinearity with variable exponent. For the equation \eqref{eq:EL}, this type of additional regularity on $p$ is not necessary.

The second main result is the H\"older continuity of bounded weak solutions to \eqref{eq:EL}.

\begin{theorem} \label{thm:Holder}
Let $\Omega$ be a bounded domain in $\mathbb{R}^n$. Assume that $p$ satisfies \eqref{eq:P1} and \eqref{eq:P2} in $\Omega$. If $u \in W^{s, p(\cdot,\cdot)}(\mathbb{R}^n)$ is a weak solution to \eqref{eq:EL} in $\Omega$ satisfying \eqref{eq:tail}, then $u$ is locally H\"older continuous in $\Omega$.
\end{theorem}

\Cref{thm:Holder} follows from a growth lemma. In order to prove the growth lemma, we need two ingredients: an improved Caccioppoli-type estimate and a fractional De Giorgi isoperimetric-type inequality. The Caccioppoli-type estimate we establish to prove \Cref{thm:local_boundedness} is actually an improved version, which was first introduced by Cozzi \cite{Coz17} for the fractional $p$-Laplacian with $p$ constant (see also \cite{CCV11}). Our Caccioppoli-type estimate not only makes it possible to use the fractional De Giorgi isoperimetric-type inequality as in \cite{Coz17}, but also takes variable exponents into account. 

The proof of \Cref{thm:Holder} is significantly different in the De Giorgi iteration from the one for the case of the constant exponent. As in the proof of \Cref{thm:local_boundedness}, we also encounter different exponents involving $p_+$ and $p_-$ in the De Giorgi iteration. However, this mismatch of exponents causes a more serious problem when we investigate the modulus of continuity of weak solutions. We will see that the assumption \eqref{eq:P1} on $p$ solves this problem. 

Another difference is that the variable exponent in the nonlocal tail term affects the iteration as well. This difficulty does not exist in the local variational problems with variable exponent as well as the nonlocal problem with constant exponent. The variable exponent in mixed regions, which appears in the nonlocal tail, interacts with the variable exponent in local terms. With this regard, the assumption \eqref{eq:P2} on $p$ is required.

Similar results to ours have recently been studied by Jihoon Ok using a different approach, see \cite{Ok21}.

The authors wish to thank Moritz Kassmann from Bielefeld University for stimulating discussions.

\subsection*{Outline}

The article is organized as follows. In \Cref{sec:preliminaries}, we recall the variable exponents Lebesgue spaces, fractional Sobolev spaces with variable exponents, and fractional Sobolev embedding theorems. \Cref{sec:Caccioppoli} is devoted to the proof of the improved Caccioppoli-type estimate with variable exponent, which will be used in the proofs of local boundedness and H\"older regularity for weak solutions. In \Cref{sec:local_boundedness}, we actually prove a stronger assertion \Cref{thm:local_boundedness2} than \Cref{thm:local_boundedness}, which provides a quantitative local estimate on the supremum of weak subsolutions. Finally, we prove \Cref{thm:Holder} in \Cref{sec:Holder} by establishing a growth lemma. This is proved by using the improved Caccioppoli-type estimate and the isoperimetric-type inequality.

\section{Preliminaries} \label{sec:preliminaries}

In this section, we briefly review the variable exponent Lebesgue spaces and fractional Sobolev spaces with variable exponents. Furthermore, we recall the fractional Sobolev embedding theorems for the constant exponent case.

\subsection{Variable exponents Lebesgue spaces} \label{sec:Lebesgue}

Let $\Omega \subset \mathbb{R}^n$ be an open set and let $p: \Omega \to \mathbb{R}$ be a measurable function satisfying
\begin{equation*}
1 < \inf_{x\in\Omega} p(x) \leq \sup_{x\in\Omega} p(x) < +\infty.
\end{equation*}
We define the variable exponent Lebesgue spaces
\begin{equation*}
L^{p(\cdot)}(\Omega) = \left\lbrace u: \Omega \to \mathbb{R} ~\text{measurable}: \varrho_{L^{p(\cdot)}(\Omega)}(u/\lambda) < +\infty ~\text{for some}~ \lambda > 0 \right\rbrace
\end{equation*}
endowed with the norm 
\begin{equation*}
\|u\|_{L^{p(\cdot)}(\Omega)} = \inf \left\lbrace \lambda > 0: \varrho_{L^{p(\cdot)}(\Omega)}(u/\lambda) \leq 1 \right\rbrace,
\end{equation*}
where
\begin{equation*}
\varrho_{L^{p(\cdot)}(\Omega)}(u) = \int_{\Omega} |u(x)|^{p(x)}\,\mathrm{d}x.
\end{equation*}
It is well known that $L^{p(\cdot)}(\Omega)$ is a Banach space, see \cite{KR91,FZ01,DHHR11} for instance. Let us collect some useful inequalities for later use.

\begin{lemma} \label{lem:modular_Lp}
\cite[Theorem 1.3]{FZ01} Let $u \in L^{p(\cdot)}(\Omega)$ and $p_\pm = p_\pm(\Omega)$, then
\begin{enumerate}[(i)]
\item
$\|u\|_{L^{p(\cdot)}(\Omega)} > 1~ (=1; <1)$ if and only if $\varrho_{L^{p(\cdot)}(\Omega)} > 1~ (=1; <1)$;
\item
if $\|u\|_{L^{p(\cdot)}(\Omega)}\geq 1$, then $\|u\|_{L^{p(\cdot)}(\Omega)}^{p_-} \leq \varrho_{L^{p(\cdot)}(\Omega)}(u) \leq \|u\|_{L^{p(\cdot)}(\Omega)}^{p_+}$;
\item
if $\|u\|_{L^{p(\cdot)}(\Omega)} \leq 1$, then $\|u\|_{L^{p(\cdot)}(\Omega)}^{p_+} \leq \varrho_{L^{p(\cdot)}(\Omega)}(u) \leq \|u\|_{L^{p(\cdot)}(\Omega)}^{p_-}$.
\end{enumerate}
\end{lemma}

\begin{lemma} \label{lem:Holder}
\cite[Theorem 2.1]{KR91} For every $u \in L^{p(\cdot)}(\Omega)$ and $v \in L^{p'(\cdot)}(\Omega)$, it holds that
\begin{equation*}
\int_{\Omega} |u(x)v(x)| \,\mathrm{d}x \leq 2 \|u\|_{L^{p(\cdot)}(\Omega)} \|v\|_{L^{p'(\cdot)}(\Omega)},
\end{equation*}
where $1/p(x) + 1/p'(x) = 1$.
\end{lemma}

See \cite{KR91,FZ01,DHHR11,CUF13} for more properties of the variable exponent Lebesgue spaces.

\subsection{Fractional Sobolev spaces with variable exponents} \label{sec:Sobolev}

The fractional Sobolev spaces with variable exponents were first introduced recently by Kaufmann, Rossi, and Vidal \cite{KRV17}, and have been studied in different contexts. See \cite{DPR17,Bah18,BR18,AHKC19,ABS19,BB19,HK19,ABS20,BRW20,BA20,CGA20,ZZ20,ASSA21,ABSS21,BT21,HK21,ZAF21} and references therein. Note that the Triebel--Lizorkin spaces with variable smoothness and integrability have been introduced in \cite{DHR09}, which are isomorphic to $W^{k, p(\cdot)}(\mathbb{R}^n)$ if $k \in \mathbb{N} \cup \lbrace 0 \rbrace$ respectively, the variable exponent Bessel potential space $\mathcal{L}^{\alpha, p(\cdot)}(\mathbb{R}^n)$ for $\alpha > 0$ under suitable assumptions on $p$. In the scope of this paper, we will focus on the fractional Sobolev spaces with variable exponents introduce in \cite{KRV17}.

In this section, let $\Omega$ be a bounded Lipschitz domain in $\mathbb{R}^n$ or $\Omega = \mathbb{R}^n$. Let $p \in C(\overline{\Omega} \times \overline{\Omega})$ be such that $p(x,y) = p(y,x)$ and
\begin{equation*}
1 < p_-(\Omega \times \Omega) \leq p_+(\Omega \times \Omega) < +\infty,
\end{equation*}
and define $\bar{p}(x) = p(x,x)$. For $s \in (0,1)$, the fractional Sobolev space with variable exponents is defined as
\begin{equation*}
W^{s, p(\cdot,\cdot)}(\Omega) := \left\lbrace u \in L^{\bar{p}(\cdot)}(\Omega): \varrho_{W^{s,p(\cdot,\cdot)}(\Omega)}(u/\lambda) < +\infty ~\text{for some}~ \lambda > 0 \right\rbrace,
\end{equation*}
where
\begin{equation*}
\varrho_{W^{s,p(\cdot,\cdot)}(\Omega)}(u) = \int_{\Omega}\int_{\Omega} \frac{|u(x)-u(y)|^{p(x,y)}}{|x-y|^{n+sp(x,y)}} \,\mathrm{d}y\,\mathrm{d}x.
\end{equation*}
We define a seminorm
\begin{equation*}
[u]_{W^{s,p(\cdot,\cdot)}(\Omega)} = \inf \left\lbrace \lambda > 0: \varrho_{W^{s,p(\cdot,\cdot)}(\Omega)}(u/\lambda) \leq 1 \right\rbrace.
\end{equation*}
It is well known \cite{KRV17} that $W^{s, p(\cdot,\cdot)}(\Omega)$ is a Banach space with the norm
\begin{equation*}
\|u\|_{W^{s,p(\cdot,\cdot)}(\Omega)} = \|u\|_{L^{\bar{p}(\cdot)}(\Omega)} + [u]_{W^{s, p(\cdot, \cdot)}(\Omega)}.
\end{equation*}
Let us also define
\begin{equation*}
\tilde{\varrho}_{W^{s,p(\cdot,\cdot)}(\Omega)}(u) = \varrho_{L^{\bar{p}(\cdot)}(\Omega)}(u) + \varrho_{W^{s,p(\cdot,\cdot)}(\Omega)}(u)
\end{equation*}
and a norm
\begin{equation*}
|u|_{W^{s,p(\cdot,\cdot)}(\Omega)} = \inf \left\lbrace \lambda > 0: \tilde{\varrho}_{W^{s,p(\cdot,\cdot)}(\Omega)}(u/\lambda) \leq 1 \right\rbrace.
\end{equation*}
Then, it is clear that two norms $\|u\|_{W^{s,p(\cdot,\cdot)}(\Omega)}$ and $|u|_{W^{s,p(\cdot,\cdot)}(\Omega)}$ are comparable, see \cite{HK19}. It is also easy to obtain the following lemma from the definitions of the norms.

\begin{lemma} \label{lem:modular}
Let $u \in W^{s, p(\cdot,\cdot)}(\Omega)$ and $p_\pm = p_\pm(\Omega \times \Omega)$, then
\begin{enumerate}[(i)]
\item
if $[u]_{W^{s,p(\cdot,\cdot)}(\Omega)} \geq 1$, then $[u]_{W^{s,p(\cdot,\cdot)}(\Omega)}^{p_-} \leq \varrho_{W^{s,p(\cdot,\cdot)}(\Omega)}(u) \leq [u]_{W^{s,p(\cdot,\cdot)}(\Omega)}^{p_+}$;
\item
if $[u]_{W^{s,p(\cdot,\cdot)}(\Omega)} \leq 1$, then $[u]_{W^{s,p(\cdot,\cdot)}(\Omega)}^{p_+} \leq \varrho_{W^{s,p(\cdot,\cdot)}(\Omega)}(u) \leq [u]_{W^{s,p(\cdot,\cdot)}(\Omega)}^{p_-}$;
\item
if $|u|_{W^{s,p(\cdot,\cdot)}(\Omega)} \geq 1$, then $|u|_{W^{s,p(\cdot,\cdot)}(\Omega)}^{p_-} \leq \tilde{\varrho}_{W^{s,p(\cdot,\cdot)}(\Omega)}(u) \leq |u|_{W^{s,p(\cdot,\cdot)}(\Omega)}^{p_+}$;
\item
if $|u|_{W^{s,p(\cdot,\cdot)}(\Omega)} \leq 1$, then $|u|_{W^{s,p(\cdot,\cdot)}(\Omega)}^{p_+} \leq \tilde{\varrho}_{W^{s,p(\cdot,\cdot)}(\Omega)}(u) \leq |u|_{W^{s,p(\cdot,\cdot)}(\Omega)}^{p_-}$.
\end{enumerate}
\end{lemma}

Recently, the fractional Sobolev embeddings with variable exponents have been studied in \cite{KRV17,HK19,HK21}. However, the fractional Sobolev embeddings with constant exponents are sufficient for the local regularity theory with variable exponents. Let us recall the following embedding theorems for constant exponent fractional Sobolev spaces.

\begin{theorem} \label{thm:sobolev} \cite[Theorem 6.7]{DNPV12}
Let $\Omega \subset \mathbb{R}^n$ be a bounded Lipschitz domain. Let $s \in (0,1)$ and $p \in [1, n/s)$. Then there exists a constant $C = C(n, s, p, \Omega) > 0$ such that, for any $u \in W^{s, p}(\Omega)$, we have
\begin{equation*}
\|u\|_{L^q(\Omega)} \leq C \|u\|_{W^{s, p}(\Omega)}
\end{equation*}
for any $q \in [1, np/(n-sp)]$.
\end{theorem}

\begin{theorem} \label{thm:homo_sobolev} \cite[Corollary 4.9]{Coz17}
Let $s \in (0,1)$, $p \in [1, n/s)$, and $R > 0$. Let $u \in W^{s, p}_0(B_R)$ and suppose that $u=0$ on a set $\Omega_0 \subset B_R$ with $|\Omega_0| \geq \gamma |B_R|$ for some $\gamma \in (0,1]$. Then,
\begin{equation*}
\|u\|_{L^{\frac{np}{n-sp}}(B_R)} \leq C [u]_{W^{s, p}(B_R)}
\end{equation*}
for some $C = C(n, s, p, \gamma) > 0$.
\end{theorem}

\begin{theorem} \label{thm:morrey} \cite[Theorem 8.2]{DNPV12}
Let $\Omega \subset \mathbb{R}^n$ be a bounded Lipschitz domain. Let $s \in (0,1)$ and $p > n/s$. Then there exists a constant $C = C(n, s, p, \Omega) > 0$ such that, for any $u \in L^p(\Omega)$, we have
\begin{equation*}
\|u\|_{C^{\alpha}(\overline{\Omega})} \leq C \|u\|_{W^{s, p}(\Omega)},
\end{equation*}
where $\alpha = (sp-n)/p$.
\end{theorem}

\section{Caccioppoli-type estimate} \label{sec:Caccioppoli}

This section is devoted to the Caccioppoli-type estimate for weak subsolutions and supersolutions to \eqref{eq:EL}. Let us first provide the definitions of weak subsolutions and supersolutions.

\begin{definition}
A function $u \in W^{s, p(\cdot,\cdot)}(\mathbb{R}^n)$ is a weak subsolution (weak supersolution, respectively) to \eqref{eq:EL} in $\Omega$ if
\begin{equation*}
\int_{\mathbb{R}^n} \int_{\mathbb{R}^n} \frac{|u(x)-u(y)|^{p(x,y)-2}(u(x)-u(y))(\varphi(x)-\varphi(y))}{|x-y|^{n+sp(x, y)}} \,\mathrm{d}y \,\mathrm{d}x \leq 0 \quad (\geq 0, ~\text{respectively})
\end{equation*}
for every nonnegative $\varphi \in W^{s, p(\cdot,\cdot)}(\mathbb{R}^n)$ such that $\varphi = 0 ~\mathrm{a.e.}$ outside $\Omega$. A function $u \in W^{s,p(\cdot, \cdot)}(\mathbb{R}^n)$ is a weak solution to \eqref{eq:EL} in $\Omega$ if it is a weak subsolution and supersolution.
\end{definition}

The Caccioppoli-type estimate is a key ingredient for the local regularity results. This type of estimate has been established by many authors (see for instance \cite{DCKP16,KMS15,BP16,Coz17}) for the case of the fractional $p$-Laplacian with a constant $p > 1$. The main difference between Caccioppoli-type estimates for the local and nonlocal operators is that the estimate for the nonlocal operator involves a nonlocal tail term. Moreover, in \cite{Coz17}, Cozzi improved the estimate to take an isoperimetric-type inequality into account. In this section, we generalize Cozzi's estimate to the fractional $p(\cdot,\cdot)$-Laplacian.

\begin{theorem} \label{thm:Caccioppoli}
Let $\Omega \subset \R^n$ be a bounded domain. Let $u \in W^{s,p(\cdot, \cdot)}(\mathbb{R}^n)$ be a weak subsolution to \eqref{eq:EL} in $\Omega$. Then, for any $B_r(x_0) \Subset B_R(x_0) \subset \Omega$ and any $k \in \mathbb{R}$,
\begin{equation} \label{eq:Caccioppoli}
\begin{split}
&\varrho_{W^{s,p(\cdot,\cdot)}(B_r(x_0))}(w_+) + \int_{B_r(x_0)} w_+(x) \int_{B_R(x_0)} \frac{w_-(y)^{p(x,y)-1}}{|x-y|^{n+sp(x,y)}} \,\mathrm{d}y \,\mathrm{d}x \\
&\leq C \int_{B_R(x_0)} \int_{B_R(x_0)} \left| \frac{w_+(x)}{R-r} \right|^{p(x,y)} \frac{\mathrm{d}y \,\mathrm{d}x}{|x-y|^{n-(1-s)p(x,y)}} \\
&\quad + C \left( \sup_{x \in B_{\frac{R+r}{2}}(x_0)} \int_{\R^n \setminus B_R(x_0)} \frac{w_+(y)^{p(x,y)-1}}{|y-x_0|^{n+sp(x,y)}} \left( \frac{2R}{R-r} \right)^{n+sp(x,y)} \, \d y \right) \int_{B_R(x_0)} w_+(x) \,\mathrm{d}x,
\end{split}
\end{equation}
where $w_\pm := (u-k)_\pm$. The constant $C$ depends only on $p_\pm(B_R(x_0) \times B_R(x_0))$.
\end{theorem}

\begin{remark}
If $u \in W^{s,p(\cdot, \cdot)}(\mathbb{R}^n)$ is a weak supersolution to \eqref{eq:EL} in $\Omega$, then \eqref{eq:Caccioppoli} holds with $w_+$ and $w_-$ replaced by $w_-$ and $w_+$, respectively.
\end{remark}

In order to prove \Cref{thm:Caccioppoli}, we need an algebraic inequality. Recall that, in the case of $p(\cdot)$-Laplacian with $1 < p_- \leq p(x) \leq p_+ < \infty$, the inequalities
\begin{equation} \label{eq:local_ineq}
\begin{split}
|Dw|^{p(x)-2} Dw \cdot D(w\eta^{p_+})
&\geq |Dw|^{p(x)} \eta^{p_+} - p_+ |Dw|^{p(x)-1} \eta^{p_+ -1} w |D\eta| \\
&\geq |Dw|^{p(x)} \eta^{p_+} - \frac{1}{2} |Dw|^{p(x)} \eta^{(p_+-1)\frac{p(x)}{p(x)-1}} - C w^{p(x)} |D\eta|^{p(x)} \\
&\geq \frac{1}{2} |Dw|^{p(x)} \eta^{p_+} - C w^{p(x)} |D\eta|^{p(x)}
\end{split}
\end{equation}
for some $C > 0$, play a crucial role for establishing Caccioppoli-type estimates (see, e.g., \cite{FZ99}). The following lemma is a discrete version of \eqref{eq:local_ineq}.

\begin{lemma} \label{lem:nonlocal_ineq}
Let $a, b \geq 0$, $\tau_1, \tau_2 \in [0,1]$, and $1 < p_- \leq p(x,y) \leq p_+ < \infty$. Then,
\begin{equation} \label{eq:nonlocal_ineq}
\begin{split}
&|a-b|^{p(x,y)-2}(a-b) \left( a\tau_1^{p_+} - b\tau_2^{p_+} \right) \\
&\geq \frac{1}{2}|a-b|^{p(x,y)} (\max\lbrace \tau_1,\tau_2 \rbrace)^{p_+} - C(\max\lbrace a,b \rbrace)^{p(x,y)} |\tau_1-\tau_2|^{p(x,y)},
\end{split}
\end{equation}
for some $C = C(p_+, p_-) > 0$.
\end{lemma}

\begin{proof}
Since
\begin{align*}
(a-b)\left( a\tau_1^{p_+}-b\tau_2^{p_+} \right) &\geq (a-b)^2\tau_1^{p_+} - b |a-b| |\tau_1^{p_+}-\tau_2^{p_+}| \quad\text{and} \\
(a-b)\left( a\tau_1^{p_+}-b\tau_2^{p_+} \right) &\geq (a-b)^2\tau_2^{p_+} - a |a-b| |\tau_1^{p_+}-\tau_2^{p_+}|,
\end{align*}
we have
\begin{equation} \label{eq:alg_ineq}
(a-b)\left( a\tau_1^{p_+}-b\tau_2^{p_+} \right) \geq (a-b)^2 (\max\lbrace \tau_1,\tau_2 \rbrace)^{p_+} - \max\lbrace a, b \rbrace |a-b| |\tau_1^{p_+}-\tau_2^{p_+}|.
\end{equation}
By convexity of the function $f(\tau) = \tau^{p_+}$, 
\begin{equation} \label{eq:convexity}
\begin{split}
|\tau_1^{p_+} - \tau_2^{p_+}|
&\leq \max\lbrace f'(\tau_1), f'(\tau_2) \rbrace |\tau_1-\tau_2| \\
&\leq {p_+} |\tau_1-\tau_2| (\max\lbrace \tau_1,\tau_2 \rbrace)^{{p_+}-1}.
\end{split}
\end{equation}
Thus, it follows from \eqref{eq:alg_ineq}, \eqref{eq:convexity}, and Young's inequality that
\begin{align*}
&|a-b|^{p(x,y)-2}(a-b) \left( a\tau_1^{p_+} - b\tau_2^{p_+} \right) \\
&\geq |a-b|^{p(x,y)} (\max\lbrace \tau_1,\tau_2 \rbrace)^{p_+} - p_+ \frac{p(x,y)-1}{p(x,y)} \varepsilon^{\frac{p(x,y)}{p(x,y)-1}} |a-b|^{p(x,y)} (\max\lbrace \tau_1,\tau_2 \rbrace)^{(p_+ -1)\frac{p(x,y)}{p(x,y)-1}} \\
&\quad - \frac{p_+}{p(x,y)} \varepsilon^{-p(x,y)} (\max\lbrace a, b \rbrace)^{p(x,y)} |\tau_1-\tau_2|^{p(x,y)} \\
&\geq \left( 1 - p_+ \varepsilon^{p_+/(p_+-1)} \right) |a-b|^{p(x,y)} (\max\lbrace \tau_1,\tau_2 \rbrace)^{p_+} - \frac{p_+}{p_-}\varepsilon^{-p_+} (\max\lbrace a, b \rbrace)^{p(x,y)} |\tau_1-\tau_2|^{p(x,y)}.
\end{align*}
Taking $\varepsilon = (1/(2p_+))^{(p_+-1)/p_+}$, we obtain \eqref{eq:nonlocal_ineq} with $C = \frac{p_+}{p_-} (2p_+)^{p_+-1}$.
\end{proof}

\begin{proof} [Proof of \Cref{thm:Caccioppoli}]
In this proof, every ball is centered at $x_0$. Let $\eta$ be a cut-off function satisfying $\eta \in [0,1]$, $\supp \eta \subset B_{\frac{R+r}{2}} \subset B_R$, $\eta \equiv 1$ in $B_r$, and $|D\eta| \leq 4/(R-r)$. Let $p_\pm = p_\pm(B_R \times B_R)$. We first assume that $u \in L^{\infty}(B_{2R})$, then $\varphi(x) = w_+(x)\eta(x)^{p_+} \in W^{s,p(\cdot, \cdot)}(\mathbb{R}^n)$ by \cite[Lemma 4.1]{Ok21}. Applying the definition of weak subsolutions with the test function $\varphi$, we have
\begin{equation} \label{eq:I_1+I_2}
\begin{split}
0
&\geq \int_{B_R} \int_{B_R} \frac{|u(x)-u(y)|^{p(x,y)-2}(u(x)-u(y))(\varphi(x) - \varphi(y))}{|x-y|^{n+sp(x,y)}} \,\d y \,\d x \\
&\quad + 2 \int_{B_R} \int_{\mathbb{R}^n \setminus B_R} \frac{|u(x)-u(y)|^{p(x,y)-2}(u(x)-u(y)) w_+(x)\eta(x)^{p_+}}{|x-y|^{n+sp(x,y)}} \,\d y \, \d x =: I_1 + I_2.
\end{split}
\end{equation}
It is easy to see that
\begin{equation} \label{eq:I_1_1}
\begin{split}
&|u(x)-u(y)|^{p(x,y)-2}(u(x)-u(y))(\varphi(x) - \varphi(y)) \\
&\geq |w_+(x)-w_+(y)|^{p(x,y)-2}(w_+(x)-w_+(y))(w_+(x)\eta(x)^{p_+} - w_+(y) \eta(y)^{p_+}),
\end{split}
\end{equation}
as in the proof of \cite[Lemma 1.4]{DCKP16}. Moreover, by \Cref{lem:nonlocal_ineq} with $a=w_+(x)$, $b=w_+(y)$, $\tau_1=\eta(x)$, and $\tau_2=\eta(y)$, we obtain
\begin{equation} \label{eq:I_1_2}
\begin{split}
&|w_+(x)-w_+(y)|^{p(x,y)-2}(w_+(x)-w_+(y))(w_+(x)\eta(x)^{p_+} - w_+(y) \eta(y)^{p_+}) \\
&\geq \frac{1}{2}|w_+(x)-w_+(y)|^{p(x,y)} (\max\lbrace \eta(x), \eta(y) \rbrace)^{p_+} \\
&\quad - C(\max\lbrace w_+(x), w_+(y) \rbrace)^{p(x,y)} |\eta(x)-\eta(y)|^{p(x,y)}
\end{split}
\end{equation}
for all $x, y \in B_R$, where $C = C(p_+, p_-)>0$. On the other hand, it is obvious that
\begin{equation} \label{eq:I_1_3}
|u(x)-u(y)|^{p(x,y)-2}(u(x)-u(y))(\varphi(x) - \varphi(y)) = 0
\end{equation}
when $u(x), u(y) \leq k$. Furthermore, if $u(x) > k$ and $
u(y) \leq k$, then
\begin{equation} \label{eq:I_1_4}
\begin{split}
&|u(x)-u(y)|^{p(x,y)-2}(u(x)-u(y))(\varphi(x) - \varphi(y)) \\
&\geq c \left((w_+(x)-w_+(y))^{p(x,y)} + w_+(x) w_-(y)^{p(x,y)-1} \right) \eta^{p_+}(x)
\end{split}
\end{equation}
by a similar argument as in \cite[Proposition 8.5]{Coz17}, where $c = 2^{p_--2} \land 1$. Therefore, combining \eqref{eq:I_1_1}–\eqref{eq:I_1_4} and using the symmetry of $p(x,y)$, we estimate $I_1$ by
\begin{equation} \label{eq:I_1_Caccioppoli}
\begin{split}
I_1
&= \left( \int_{A_{k,R}^+} \int_{A_{k,R}^+} + 2\int_{A_{k,R}^+} \int_{B_R \setminus A_{k,R}^+} \right) \frac{|u(x)-u(y)|^{p(x,y)-2}(u(x)-u(y))(\varphi(x) - \varphi(y))}{|x-y|^{n+sp(x,y)}} \,\d y \,\d x \\ 
&\geq c \int_{B_r} \int_{B_r} \frac{|w_+(x)-w_+(y)|^{p(x,y)}}{|x-y|^{n+sp(x,y)}} \,\d y \,\d x + 2c \int_{B_r} w_+(x) \int_{B_R} \frac{w_-(y)^{p(x,y)-1}}{|x-y|^{n+sp(x,y)}} \,\mathrm{d}y \,\mathrm{d}x \\
&\quad - C \int_{B_R} \int_{B_R} \frac{w_+(x)^{p(x,y)} |\eta(x)-\eta(y)|^{p(x,y)}}{|x-y|^{n+sp(x,y)}} \,\d y \,\d x,
\end{split}
\end{equation}
where $A_{k,R}^+ = B_R \cap \lbrace u > k \rbrace$.

For $I_2$, we use the inequalities
\begin{equation*}
\begin{split}
|u(x)-u(y)|^{p(x,y)-2}(u(x)-u(y)) w_+(x)
&\geq -(u(y)-u(x))_+^{p(x,y)-1}w_+(x) \\
&\geq -w_+(y)^{p(x,y)-1} w_+(x)
\end{split}
\end{equation*}
and
\begin{equation*}
\frac{|y-x_0|}{|x-y|} \leq 1 + \frac{|x-x_0|}{|x-y|} \leq 1+\frac{R+r}{R-r} = \frac{2R}{R-r}, \quad x \in B_{\frac{R+r}{2}}, y \in \mathbb{R}^n \setminus B_R,
\end{equation*}
to obtain
\begin{equation} \label{eq:I_2_Caccioppoli}
\begin{split}
I_2
&\geq - 2 \int_{B_R} \int_{\mathbb{R}^n \setminus B_R} \frac{w_+(y)^{p(x,y)-1} w_+(x) \eta(x)^{p_+}}{|x-y|^{n+sp(x,y)}} \,\d y \,\d x \\
&\geq - 2\left( \sup_{x \in \mathrm{supp}\, \eta} \int_{\mathbb{R}^n \setminus B_R} \frac{w_+(y)^{p(x,y)-1}}{|x-y|^{n+sp(x,y)}} \,\d y \right) \int_{B_R} w_+(x) \eta(x)^{p_+} \,\mathrm{d}x \\
&\geq - 2 \left( \sup_{x \in B_{\frac{R+r}{2}}} \int_{\R^n \setminus B_R} \frac{w_+(y)^{p(x,y)-1}}{|y-x_0|^{n+sp(x,y)}} \left( \frac{2R}{R-r} \right)^{n+sp(x,y)} \, \d y \right) \int_{B_R} w_+(x) \,\mathrm{d}x.
\end{split}
\end{equation}
Therefore, \eqref{eq:Caccioppoli} follows from \eqref{eq:I_1+I_2}, \eqref{eq:I_1_Caccioppoli}, \eqref{eq:I_2_Caccioppoli}, and $|D\eta| \leq 4/(R-r)$.

The general case $u \in W^{s, p(\cdot,\cdot)}(\mathbb{R}^n)$ follows in the standard way by using truncated functions and the monotone convergence theorem.
\end{proof}

\section{Local boundedness} \label{sec:local_boundedness}

In this section, we prove \Cref{thm:local_boundedness}. The idea of the proof of the local boundedness is to fix a point $x_0 \in \Omega$ and find a small ball $B_{R/2}(x_0) \subset \Omega$ on which $u$ is bounded. We begin with \Cref{thm:local_boundedness2}, where the case $p(x_0, x_0) \leq n/s$ is covered. \Cref{thm:local_boundedness2} not only proves the local boundedness of weak subsolutions in $B_{R/2}(x_0)$ but also provides a quantitative estimate of their supremum. The proof of \Cref{thm:local_boundedness2} is based on the De Giorgi iteration technique. In the end of this section, we will prove \Cref{thm:local_boundedness} by combining \Cref{thm:local_boundedness2} and \Cref{thm:morrey}.

\begin{theorem} \label{thm:local_boundedness2}
Let $\Omega$ be a bounded domain in $\mathbb{R}^n$ and let $0<\sigma<s<1$. Let $u \in W^{s, p(\cdot,\cdot)}(\mathbb{R}^n)$ be a weak subsolution to \eqref{eq:EL} in $\Omega$ satisfying \eqref{eq:tail}. For each $x_0 \in \Omega$ with $p(x_0, x_0) \leq n/s$, there is a radius $R \in (0,1)$ such that $B_{R} = B_{R}(x_0) \subset \Omega$, $p_+ < p_-^{\ast} := \frac{np_-}{n-\sigma p_-}$, and
\begin{equation} \label{eq:bounded_above}
\begin{split}
\sup_{B_{R/2}} u 
&\leq C\max \left\lbrace \left( \fint_{B_{R}} u_+^{p_+}(x) \,\mathrm{d}x \right)^{\frac{1}{p_+}}, \left( \fint_{B_{R}} u_+^{p_+}(x) \,\mathrm{d}x \right)^{\frac{1}{p_-} \frac{q-p_-}{q-p_+}} \right\rbrace \\
&\quad + \left( \sup_{x \in B_{R}} \int_{\R^n \setminus B_{R/2}} \frac{u_+(y)^{p(x,y)-1}}{|y-x_0|^{n+sp(x,y)}} \, \d y \right)^{1/(p_+-1)} + 1
\end{split}
\end{equation}
for any $q \in (\max\lbrace p_+, \frac{n}{n-\sigma} \rbrace, p_-^{\ast})$, where $p_\pm = p_\pm(B_{R} \times B_{R})$. The constant $C$ depends on $n$, $s$, $\sigma$, $p_+(B_R \times \mathbb{R}^n)$, $p_-(B_R \times B_R)$, $q$, and $R$.
\end{theorem}

\begin{remark}
If $u \in W^{s, p(\cdot,\cdot)}(\mathbb{R}^n)$ is a weak supersolution to \eqref{eq:EL} in $\Omega$ satisfying \eqref{eq:tail} (with $u_-$), then $-u$ is a weak subsolution to \eqref{eq:EL} in $\Omega$ satisfying \eqref{eq:tail} (with $u_+$), and hence \eqref{eq:bounded_above} holds with $u$ replaced by $-u$.
\end{remark}

The Caccioppoli-type inequality and the fractional Sobolev inequality are crucial tools for the De Giorgi iteration. One can make use of the fractional Sobolev inequality with variable exponent developed in \cite{HK19}, but it requires the assumption $p_+(B_{R}(x_0) \times B_{R}(x_0)) < n/s$, which is stronger than the assumption made in \Cref{thm:local_boundedness2}, namely $p(x_0, x_0) \leq n/s$. Thus, we will use the fractional Sobolev inequality with constant exponent (\Cref{thm:sobolev}).

For local variational problems, we have the continuous embedding $W^{1,p(\cdot)}(\Omega) \hookrightarrow W^{1,p_-}(\Omega)$ by a simple inequality $\int |Du|^{p_-} \,\mathrm{d}x \leq \int (|Du|^{p(x)}+1)\,\mathrm{d}x$. However, a similar continuous embedding $W^{s, p(\cdot, \cdot)}(\Omega) \hookrightarrow W^{s, p_-}(\Omega)$ is not available. Instead, we prove the following lemma, which shows a continuous embedding into a larger space with smaller orders of differentiability $\sigma < s$ and integrability $q < p_-$. This lemma is a generalization of \cite[Lemma 4.6]{Coz17}.

\begin{lemma} \label{lem:embedding}
Let $\Omega' \subset \Omega \subset \mathbb{R}^n$ be two bounded measurable sets with $d:= \mathrm{diam}(\Omega) \leq 1$. Let $1 \leq q < p_- \leq p(x,y) \leq p_+$ and $0<\sigma < s <1$, where $p_\pm = p_\pm(\Omega \times \Omega)$, then $W^{s, p(\cdot,\cdot)}(\Omega) \hookrightarrow W^{\sigma, q}(\Omega)$. In particular, for any $u \in W^{s, p(\cdot, \cdot)}(\Omega)$,
\begin{equation*}
\begin{split}
&\left( \int_{\Omega} \int_{\Omega'} \frac{|u(x)-u(y)|^q}{|x-y|^{n+\sigma q}} \,\mathrm{d}y \,\mathrm{d}x \right)^{1/q} \\
&\leq C \max \left\lbrace \left( |\Omega'| d^{(s-\sigma) \frac{p_+ q}{p_+-q}} \right)^{\frac{p_+-q}{p_+q}}, \left( |\Omega'| d^{(s-\sigma) \frac{p_+ q}{p_+-q}} \right)^{\frac{p_--q}{p_-q}} \right \rbrace [u]_{W^{s,p(\cdot, \cdot)}(\Omega)},
\end{split}
\end{equation*}
where $C = C(n, s, \sigma, p_+, p_-, q) > 0$.
\end{lemma}

\begin{proof}
We define
\begin{equation*}
U(x,y) := \frac{|u(x)-u(y)|^q}{|x-y|^{n+\sigma q}},
\end{equation*}
then
\begin{equation*}
U(x,y) = \frac{|u(x)-u(y)|^q}{|x-y|^{(n+sp(x,y))\frac{q}{p(x,y)}}} \frac{1}{|x-y|^{n\frac{p(x,y)-q}{p(x,y)}- (s-\sigma)q}} =: V(x,y) W(x,y).
\end{equation*}
Thus, by \Cref{lem:Holder} we obtain
\begin{equation*}
\begin{split}
\int_{\Omega} \int_{\Omega'} \frac{|u(x)-u(y)|^q}{|x-y|^{n+\sigma q}} \,\mathrm{d}y \,\mathrm{d}x 
&= \|U\|_{L^1(\Omega \times \Omega')} \\
&\leq 2\|V\|_{L^{\frac{p(\cdot,\cdot)}{q}}(\Omega \times \Omega')} \|W\|_{L^{\frac{p(\cdot,\cdot)}{p(\cdot,\cdot)-q}}(\Omega \times \Omega')} \\
&= 2[u]_{W^{s,p(\cdot,\cdot)}(\Omega)}^q \|W\|_{L^{\frac{p(\cdot,\cdot)}{p(\cdot,\cdot)-q}}(\Omega \times \Omega')}.
\end{split}
\end{equation*}
Using \Cref{lem:modular_Lp}, we have
\begin{equation*}
\|W\|_{L^{\frac{p(\cdot,\cdot)}{p(\cdot,\cdot)-q}}(\Omega \times \Omega')} \leq \max \left\lbrace \varrho_{L^{\frac{p(\cdot,\cdot)}{p(\cdot,\cdot)-q}}(\Omega \times \Omega')}(W)^{\frac{p_+-q}{p_+}}, \varrho_{L^{\frac{p(\cdot,\cdot)}{p(\cdot,\cdot)-q}}(\Omega \times \Omega')}(W)^{\frac{p_--q}{p_-}} \right\rbrace.
\end{equation*}
Since $d \leq 1$, 
\begin{equation*}
\begin{split}
\varrho_{L^{\frac{p(\cdot,\cdot)}{p(\cdot,\cdot)-q}}(\Omega \times \Omega')}(W)
&= \int_{\Omega'} \int_{\Omega} \frac{|x-y|^{(s-\sigma)\frac{p(x,y)q}{p(x,y)-q}}}{|x-y|^n} \,\mathrm{d}y \,\mathrm{d}x \\
&\leq \int_{\Omega'} \int_{\Omega} \frac{|x-y|^{(s-\sigma)\frac{p_+q}{p_+-q}}}{|x-y|^n} \,\mathrm{d}y \,\mathrm{d}x \\
&\leq \frac{p_+-q}{(s-\sigma)p_+q} |\mathbb{S}^{n-1}| |\Omega'| d^{(s-\sigma)\frac{p_+q}{p_+-q}}.
\end{split}
\end{equation*}
Therefore, combining the previous estimates finishes the proof.
\end{proof}

As mentioned before, we will prove \Cref{thm:local_boundedness2} by using the De Giorgi iteration technique. For this purpose, we need the following lemma.

\begin{lemma} \label{lem:iteration}
Suppose that a sequence $\lbrace Y_j \rbrace_{j=0}^{\infty}$ of nonnegative numbers satisfies the recursion relation
\begin{equation*}
Y_{j+1} \leq C b^{j} \max \left\lbrace Y_j^{1+\beta_1}, Y_j^{1+\beta_2}, \dots, Y_j^{1+\beta_N} \right\rbrace
\end{equation*}
for some constants $C \geq 1$, $b > 1$, $N \in \mathbb{N}$, and $\beta_1 \geq \beta_2 \geq \dots \geq \beta_N > 0$. If
\begin{equation} \label{eq:Y_0}
Y_0 \leq \min \left\lbrace C^{-\frac{1}{\beta_N}} b^{-\frac{1}{\beta_N^2}}, C^{-\frac{1}{\beta_1}} \right\rbrace,
\end{equation}
then
\begin{equation} \label{eq:Y_j}
Y_j \leq C^{-\frac{1}{\beta_N}} b^{-\frac{1}{\beta_N^2} - \frac{j}{\beta_N}} \quad\text{for all}~ j \geq 0,
\end{equation}
and, consequently, $Y_j \to 0$ as $j \to \infty$.
\end{lemma}

\begin{proof}
When $Y_0 \leq 1$, one can easily prove by induction that
\begin{equation*}
Y_j \leq C^{\frac{(1+\beta_N)^j-1}{\beta_N}} b^{\frac{(1+\beta_N)^j-1}{\beta_N^2}-\frac{j}{\beta_N}} Y_0^{(1+\beta_N)^j} \leq 1, \quad \text{for all}~ j \geq 0,
\end{equation*}
under the assumption $Y_0 \leq C^{-\frac{1}{\beta_N}} b^{-\frac{1}{\beta_N^2}}$. This yields \eqref{eq:Y_j}.

When $Y_0 \geq 1$, a similar argument shows that
\begin{equation*}
Y_j \leq C^{\frac{(1+\beta_N)^j-1}{\beta_N}} b^{\frac{(1+\beta_N)^j-1}{\beta_N^2}-\frac{j}{\beta_N}} Y_0^{(1+\beta_1)(1+\beta_N)^{j-1}} \leq 1, \quad \text{for all}~ j \geq 1,
\end{equation*}
under the assumption \eqref{eq:Y_0}. This also proves \eqref{eq:Y_j}.
\end{proof}

We are now in a position to prove \Cref{thm:local_boundedness2} by using \Cref{thm:Caccioppoli}, \Cref{lem:embedding}, and \Cref{lem:iteration}.

\begin{proof} [Proof of \Cref{thm:local_boundedness2}]
Let $x_0 \in \Omega$ with $p(x_0,x_0) \leq n/s$. By the continuity of $p$, we can take $R \in (0,1/2)$ sufficiently small such that $B_{R}(x_0) \subset \Omega$ and
\begin{equation} \label{eq:p_close}
p_+ < p_-^{\ast} := \frac{np_-}{n-\sigma p_-},
\end{equation}
where $p_\pm = p_\pm(B_{R}(x_0) \times B_{R}(x_0))$. Note that $\sigma p_- < s p_- \leq n$.

We fix $k \in \R$ and $\tilde{k} \in \R^+$. In order to use the De Giorgi iteration, we set for each $j \in \mathbb{N} \cup \lbrace 0 \rbrace$
\begin{equation*}
r_j = \frac{1}{2}(1+2^{-j})R, \quad \tilde{r}_j = \frac{r_j+r_{j+1}}{2}, \quad k_j = k+(1-2^{-j}) \tilde{k}, \quad\text{and}\quad \tilde{k}_j = \frac{k_{j+1}+k_j}{2},
\end{equation*}
and define
\begin{equation*}
w_j = (u-k_j)_+ \quad\text{and}\quad \tilde{w}_j = (u-\tilde{k}_j)_+.
\end{equation*}
For simplicity, we write $B_j = B(x_0, r_j)$ and $\tilde{B}_j = B(x_0, \tilde{r}_j)$.

By \eqref{eq:p_close}, we can choose a constant $q$ such that
\begin{equation} \label{eq:pq_close}
\max\left\lbrace p_+, \frac{n}{n-\sigma} \right\rbrace < q < p_-^{\ast}.
\end{equation}
Then, $q = t^{\ast} = \frac{nt}{n-\sigma t}$ for some $1 < t < p_- \leq n/\sigma$. By applying \Cref{thm:sobolev} to $\tilde{w}_j$ in $\tilde{B}_j$, we have
\begin{equation*}
\|\tilde{w}_j\|_{L^{q}(\tilde{B}_j)} \leq C \|\tilde{w}_j\|_{W^{\sigma, t}(\tilde{B}_j)}
\end{equation*}
for some $C > 0$ depending on $\tilde{r}_j$. Since $\tilde{r}_j \in [R/2, R]$, we may assume that $C$ depends on $R$, but not on $j$, with a possibly larger constant $C$. Since the quantities $n$, $s$, $\sigma$, $p_+(B_R \times \mathbb{R}^n)$, $p_-(B_R \times B_R)$, $q$, and $R$ are not important for the iteration, we will absorb these quantities into constants $C$. Moreover, using \Cref{lem:embedding} and \Cref{lem:Holder}, we have
\begin{equation*}
\|\tilde{w}_j \|_{W^{\sigma, t}(\tilde{B}_j)} \leq C \|\tilde{w}_j \|_{W^{s,p(\cdot,\cdot)}(\tilde{B}_j)}
\end{equation*}
for some $C = C(n, s, \sigma, p_+, p_-, q, R) > 0$.
By \Cref{lem:modular},
\begin{equation} \label{eq:max}
\begin{split}
\|\tilde{w}_j\|_{W^{s,p(\cdot,\cdot)}(\tilde{B}_j)}
&\leq 2|\tilde{w}_j|_{W^{s,p(\cdot,\cdot)}(\tilde{B}_j)} \\
&\leq 2\max \left\lbrace \tilde{\varrho}_{W^{s, p(\cdot,\cdot)}(\tilde{B}_j)}(\tilde{w}_j)^{1/p_-}, \tilde{\varrho}_{W^{s, p(\cdot,\cdot)}(\tilde{B}_j)}(\tilde{w}_j)^{1/p_+}\right\rbrace.
\end{split}
\end{equation}
We set $A_{h, r}^+ = B_r \cap \lbrace u > h \rbrace$ and 
\begin{equation*}
Y_j = \fint_{B_j} w_j^{p_+}(x) \,\mathrm{d}x,
\end{equation*}
and estimate $\tilde{\varrho}_{W^{s, p(\cdot,\cdot)}(\tilde{B}_j)}(\tilde{w}_j)$ in terms of $j$ and $Y_j$.

By \Cref{thm:Caccioppoli}, we have
\begin{equation*}
\begin{split}
&\tilde{\varrho}_{W^{s, p(\cdot,\cdot)}(\tilde{B}_j)}(\tilde{w}_j) \\
&\leq C 2^{jp_+} \int_{B_j} \int_{B_j} \frac{\tilde{w}_j(x)^{p(x,y)}}{|x-y|^{n+sp(x,y)}} |x-y|^{p(x,y)} \,\mathrm{d}y \,\mathrm{d}x \\
&\quad + C \left( \sup_{x \in B(x_0,\frac{1}{2}(r_j+\tilde{r}_j))} \int_{\R^n \setminus B_j} \frac{\tilde{w}_j(y)^{p(x,y)-1}}{|y-x_0|^{n+sp(x,y)}} \left( \frac{4r_j}{r_j-r_{j+1}} \right)^{n+sp(x,y)} \, \d y \right) \int_{B_j} \tilde{w}_j(x) \,\mathrm{d}x \\
&=: I_1 + I_2.
\end{split}
\end{equation*}
Since $\tilde{w}_j = 0$ on $B_j \setminus A_{\tilde{k}_j, r_j}^+$, $\tilde{w}_j \leq w_j$, and $R < 1/2$, we estimate $I_1$ as follows:
\begin{equation*}
\begin{split}
I_1
&\leq C 2^{jp_+} \int_{A_{\tilde{k}_j, r_j}^+} \int_{B_j} (w_j(x)^{p_+}+1) \frac{|x-y|^{(1-s)p_-}}{|x-y|^n} \,\mathrm{d}y \,\mathrm{d}x \\
&\leq C 2^{jp_+} \left( \int_{A_{\tilde{k}_j, r_j}^+} w_j(x)^{p_+} \,\d x + |A_{\tilde{k}_j, r_j}^+| \right) \\
&\leq C 2^{jp_+} \left(Y_j + |A_{\tilde{k}_j, r_j}^+| \right).
\end{split}
\end{equation*}
For $I_2$, we use $p(x,y) \leq p_+(B_R \times \mathbb{R}^n)$ and $(\tilde{k}_j - k_j)^{p_+-1} \tilde{w}_j \leq w_j^{p_+}$. Then,
\begin{equation*}
\begin{split}
I_2
&\leq C 2^{j(n+sp_+(B_R\times\mathbb{R}^n))} \left( \sup_{x \in B_{R}} \int_{\R^n \setminus B_{R/2}} \frac{\tilde{w}_j(y)^{p(x,y)-1}}{|y-x_0|^{n+sp(x,y)}} \, \d y \right) \int_{B_j} \frac{w_j^{p_+}(x)}{(\tilde{k}_j-k_j)^{p_+ -1}} \,\mathrm{d}x \\
&\leq C 2^{j(n+sp_+(B_R \times \mathbb{R}^n))} \left( \sup_{x \in B_{R}} \int_{\R^n \setminus B_{R/2}} \frac{w_0(y)^{p(x,y)-1}}{|y-x_0|^{n+sp(x,y)}} \, \d y \right) \left( \frac{2^{j+2}}{\tilde{k}} \right)^{p_+-1} \int_{A_{k_j, r_j}^+} w_j^{p_+} \,\mathrm{d}x \\
&\leq C 2^{j(n+2p_+(B_R \times \mathbb{R}^n))} \frac{T}{\tilde{k}^{p_+-1}} Y_j,
\end{split}
\end{equation*}
where
\begin{equation*}
T = \sup_{x \in B_{R}} \int_{\R^n \setminus B_{R/2}} \frac{u_+(y)^{p(x,y)-1}}{|y-x_0|^{n+sp(x,y)}} \, \d y.
\end{equation*}
Combining the estimates above and using
\begin{equation*}
|A_{\tilde{k}_j, r_j}^+| \leq \frac{1}{(\tilde{k}_j - k_j)^{p_+}} \int_{A_{\tilde{k}_j, r_j}} w_j^{p_+} \,\d x \leq C \left( \frac{2^{j}}{\tilde{k}} \right)^{p_+} Y_j,
\end{equation*}
yield that
\begin{equation*}
\tilde{\varrho}_{W^{s, p(\cdot,\cdot)}(\tilde{B}_j)}(\tilde{w}_j) \leq C 2^{j(n+2p_+(B_R \times \mathbb{R}^n)-1)} \left( 1+ \frac{1}{\tilde{k}^{p_+}} + \frac{T}{\tilde{k}^{p_+-1}} \right) Y_j.
\end{equation*}
Assuming
\begin{equation} \label{eq:tilde_k}
\tilde{k} \geq T^{1/(p_+-1)} + 1,
\end{equation}
we arrive at
\begin{equation} \label{eq:modulus}
\tilde{\varrho}_{W^{s, p(\cdot,\cdot)}(\tilde{B}_j)}(\tilde{w}_j) \leq C 2^{j(n+2p_+(B_R \times \mathbb{R}^n))} Y_j.
\end{equation}

On the other hand, recalling that \eqref{eq:pq_close} holds, we have $\tilde{w}_j^{q} \geq (k_{j+1}-\tilde{k}_j)^{q-p_+} w_{j+1}^{p_+}$, and hence
\begin{equation} \label{eq:q_lower}
\begin{split}
\|\tilde{w}_j\|_{L^{q}(\tilde{B}_j)}^{q} \geq \|\tilde{w}_j\|_{L^{q}(B_{j+1})}^{q} \geq c(2^{-j}\tilde{k})^{q-p_+} Y_{j+1}.
\end{split}
\end{equation}
Therefore, from \eqref{eq:max}, \eqref{eq:modulus}, and \eqref{eq:q_lower}, we deduce 
\begin{equation*}
Y_{j+1} \leq C \tilde{k}^{p_+-q} b^j \max \left\lbrace Y_j^{1+\beta_1}, Y_j^{1+\beta_2} \right\rbrace,
\end{equation*}
where $\beta_1 = q/p_- -1 > 0$, $\beta_2 = q/p_+ - 1 > 0$, $b = 2^{q-p_+ + (n+2p_+(B_R \times \mathbb{R}^n))q/p_-}$, and $C > 0$ is a constant depending only on $n$, $s$, $\sigma$, $p_+$, $p_-$, $q$, and $R$. 

By \Cref{lem:iteration}, if
\begin{equation} \label{eq:J_0}
Y_0 \leq \min \left\lbrace (C\tilde{k}^{p_+-q})^{-\frac{1}{\beta_2}} b^{-\frac{1}{\beta_2^2}}, (C\tilde{k}^{p_+-q})^{-\frac{1}{\beta_1}} \right\rbrace,
\end{equation}
then $Y_j \to 0$ as $j \to \infty$. Thus, if we take
\begin{equation} \label{eq:tilde_k_max}
\tilde{k} \geq \max \left\lbrace \left( C^{\frac{1}{\beta_2}}b^{\frac{1}{\beta_2^2}} Y_0 \right)^{\frac{1}{p_+}}, \left( C^{\frac{1}{\beta_1}} Y_0 \right)^{\frac{1}{p_-} \frac{q-p_-}{q-p_+}} \right\rbrace,
\end{equation}
then \eqref{eq:J_0} is satisfied, and hence
\begin{equation*}
\sup_{B_{R/2}} u \leq k+\tilde{k}.
\end{equation*}
Note that the choice
\begin{equation*}
\tilde{k} = C_0 \max \left\lbrace \left( \fint_{B_{R}} w_0^{p_+}(x) \,\mathrm{d}x \right)^{\frac{1}{p_+}}, \left( \fint_{B_{R}} w_0^{p_+}(x) \,\mathrm{d}x \right)^{\frac{1}{p_-} \frac{q-p_-}{q-p_+}} \right\rbrace +T^{1/(p_+-1)}+1
\end{equation*}
with
\begin{equation*}
C_0 = \max \left\lbrace \left( C^{\frac{1}{\beta_2}}b^{\frac{1}{\beta_2^2}} \right)^{\frac{1}{p_+}}, \left( C^{\frac{1}{\beta_1}} \right)^{\frac{1}{p_-} \frac{q-p_-}{q-p_+}} \right\rbrace
\end{equation*}
is in accordance with \eqref{eq:tilde_k} and \eqref{eq:tilde_k_max}. The constant $C_0$ depends on $n$, $s$, $\sigma$, $p_+(B_R \times \mathbb{R}^n)$, $p_-(B_R\times B_R)$, $q$, and $R$. We finish the proof by choosing $k=0$.
\end{proof}

Let us conclude this section with the proof of \Cref{thm:local_boundedness} by using \Cref{thm:local_boundedness2} and \Cref{thm:morrey}.

\begin{proof} [Proof of \Cref{thm:local_boundedness}]
Suppose that $u \in W^{s, p(\cdot,\cdot)}(\mathbb{R}^n)$ is a weak subsolution to \eqref{eq:EL} in $\Omega$ satisfying \eqref{eq:tail}. Let us fix $x_0 \in \Omega$. If $p(x_0, x_0) \leq n/s$, then by \Cref{thm:local_boundedness2},
\begin{equation*}
\sup_{B_{R/2}(x_0)} u < +\infty
\end{equation*}
for some $R \in (0,1)$. If $p(x_0, x_0) > n/s$, then the fact that $u \in W^{s, p(\cdot,\cdot)}(\Omega)$ implies that $u$ is bounded in a neighborhood of $x_0$. Indeed, by the continuity of $p$, we can take $R > 0$ such that $\overline{B_{R}(x_0)} \subset \Omega$ and
\begin{equation*}
p_- := p_-(B_{R}(x_0) \times B_{R}(x_0)) > \frac{n}{\sigma}
\end{equation*}
for $\sigma \in (0, s)$ sufficiently close to $s$. Let $q \in (n/\sigma, p_-)$, then by \Cref{thm:morrey}
\begin{equation*}
\|u\|_{C^\alpha(\overline{B_{R}(x_0)})} \leq \|u\|_{W^{\sigma, q}(B_{R}(x_0))}.
\end{equation*}
Moreover, by \Cref{lem:embedding} and \Cref{lem:Holder}, we obtain
\begin{equation*}
\|u\|_{W^{\sigma, q}(B_{R}(x_0))} \leq C \|u\|_{W^{s, p(\cdot,\cdot)}(B_{R}(x_0))} < +\infty.
\end{equation*}
Therefore, 
\begin{equation*}
\|u\|_{L^\infty(B_{R}(x_0))} \leq \|u\|_{C^\alpha(\overline{B_{R}(x_0)})} <+\infty.
\end{equation*}

When $u \in W^{s, p(\cdot,\cdot)}(\mathbb{R}^n)$ is a weak supersolution, $-u$ is a weak subsolution and hence $\inf_{B_{R}(x_0)} u = -\sup_{B_{R}(x_0)} (-u) > -\infty$ for some $R > 0$.
\end{proof}

\section{H\"older estimate} \label{sec:Holder}
This section is devoted to the proof of the local H\"older regularity of weak solutions to \eqref{eq:EL}. In this part of the paper, the assumptions \eqref{eq:P1} and \eqref{eq:P2} on $p$ take an important role for our analysis. The key step in establishing the local H\"older regularity is a growth lemma, see \Cref{lem:growth}. We start with an auxiliary result that is needed in the proof of the growth lemma. 
\begin{lemma} \label{lemma:ellest}
Let $B_R = B_R(x_0) \subset \mathbb{R}^n$ with $R \in (0,1)$. Let $H > 0$, $\delta\in(0,1/8]$ and $0<\sigma < s <1$. Assume that $p$ satisfies \eqref{eq:P1} and \eqref{eq:P2} in $B_R$, $H^{p_+ - p_-} \leq 2$, and $p_+ < p_-^\ast$, where $p_\pm = p_\pm(B_R \times B_R)$ and $p_-^\ast = \frac{np_-}{n-\sigma p_-}$. Let $u \in W^{s, p(\cdot,\cdot)}(\mathbb{R}^n)$ be a weak supersolution to \eqref{eq:EL} in $B_R$ such that
\begin{equation} \label{eq:assumption1}
0 \leq u \leq 2H \quad\text{in}~ B_R \quad\text{and}\quad |B_{R/2} \cap \lbrace u \geq H \rbrace| \geq \gamma |B_{R/2}|
\end{equation}
for some $\gamma \in (0,1)$. Assume $R^s \leq \delta H$ and
\begin{equation} \label{eq:assumption_tail2}
\sup_{x \in B_{3R/4}} \int_{\mathbb{R}^n \setminus B_R} \frac{u_-(y)^{p(x,y)-1}}{|y-x_0|^{n+sp(x,y)}} \,\mathrm{d}y \leq R^{-sp_+} (\delta H)^{p_+-1} + R^{-sp_-} (\delta H)^{p_--1}.
\end{equation}
Let $1 \leq q < p_- $.
Then there is a constant $C=C(n, s, \sigma, p_+(B_R \times \mathbb{R}^n), p_-(B_R \times \mathbb{R}^n), q, L) > 0$ such that for any $\ell\in[2\delta H,  H]$
\begin{equation*}
[(u-\ell)_{-}]_{W^{\sigma,q}(B_{R/2})}^q \leq C \ell^q R^{-\sigma q} \max \left\lbrace |A_{\ell, R}^-|, |A_{\ell, R}^-|^{1+\frac{q}{p_-}-\frac{q}{p_+}}, |A_{\ell, R}^-|^{1+\frac{q}{p_+} - \frac{q}{p_-}} \right\rbrace,
\end{equation*}
where $A_{\ell, R}^- = B_R \cap \lbrace u < \ell \rbrace$.
\end{lemma}

\begin{proof}
Let $\ell\in[2\delta H,  H]$. The idea of the proof is to estimate $[(u-\ell)_{-}]_{W^{\sigma,q}(B_{R/2})}^q$ using \Cref{lem:embedding}
and then applying the Caccioppoli-type inequality to estimate $\varrho_{W^{s,p(\cdot,\cdot)}(B_{R/2})}((u-\ell)_-) $. 
In the following $C>0$ denotes a constant depending on $n$, $s$, $\sigma$, $p_{+}(B_R\times\R^n)$, $p_{-}(B_R\times\R^n)$, $q$, and $L$ whose exact value is not important and might change from line to line.

Let $r=R/2$. First, by \Cref{lem:embedding}
\begin{align*}
[(u-\ell)_{-}]_{W^{\sigma,q}(B_{r})}^q &\leq C\int_{B_{r}}\int_{B_{r}} \frac{|(u(x)-\ell)_{-} - (u(y)-\ell)_{-}|^q}{|x-y|^{n+\sigma q}} \, \d x\, \d y \\
& \leq C\int_{A_{\ell,r}^{-}}\int_{B_{r}} \frac{|(u(x)-\ell)_{-} - (u(y)-\ell)_{-}|^q}{|x-y|^{n+\sigma q}} \, \d x\, \d y \\
& \leq C \max \left\lbrace \left( |A_{\ell,r}^{-}| R^{(s-\sigma) \frac{p_+q }{p_+-q}} \right)^{\frac{p_+-q}{p_+}}, \left( |A_{\ell,r}^{-}| R^{(s-\sigma) \frac{p_+q }{p_+-q}} \right)^{\frac{p_--q}{p_-}} \right \rbrace \\
&\qquad \times [(u-\ell)_{-}]_{W^{s,p(\cdot, \cdot)}(B_{r})}^q \\
& \leq C \max \left\lbrace  |A_{\ell,r}^{-}|^{\frac{p_+-q}{p_+}} R^{q(s-\sigma)} ,  |A_{\ell,r}^{-}|^{\frac{p_--q}{p_-}} R^{(s-\sigma) \frac{p_+ q(p_--q)}{(p_+-q)p_-}} \right \rbrace \\
& \qquad \times \max\left\{ \left(\varrho_{W^{s,p(\cdot,\cdot)}(B_{r})}((u-\ell)_{-})\right)^{\frac{q}{p_{+}}}, \left(\varrho_{W^{s,p(\cdot,\cdot)}(B_{r})}((u-\ell)_{-})\right)^{\frac{q}{p_{-}}} \right\}.
\end{align*}
By \Cref{thm:Caccioppoli}, we can estimate $\varrho_{W^{s,p(\cdot,\cdot)}(B_{r})}((u-\ell)_-)$ as follows:
\begin{align*}
&\varrho_{W^{s,p(\cdot,\cdot)}(B_{r})}((u-\ell)_-) + \int_{B_r} (u-\ell)_-(x) \int_{B_r} \frac{(u(y)-\ell)_+^{p(x,y)-1}}{|x-y|^{n+sp(x,y)}} \,\mathrm{d}y \,\mathrm{d}x \\
&\leq C \int_{B_R} \int_{B_R} \left| \frac{(u(x)-\ell)_-}{R-r} \right|^{p(x,y)} \frac{\mathrm{d}y \,\mathrm{d}x}{|x-y|^{n-(1-s)p(x,y)}} \\
&\quad + C \left( \sup_{x \in B_{\frac{R+r}{2}}} \int_{\R^n \setminus B_R} \frac{(u(y)-\ell)_-^{p(x,y)-1}}{|y-x_0|^{n+sp(x,y)}} \,\left( \frac{2R}{R-r} \right)^{n+sp(x,y)} \d y \right)\int_{B_R} (u(x)-\ell)_- \,\mathrm{d}x \\
& =: I_1 +  I_2.
\end{align*}
First, we consider $I_1$. By the nonnegativity of $u$ in $B_R$,
\begin{align*}
I_1 & = C \int_{ A_{\ell,R}^{-}} \int_{B_R} \left| \frac{(u(x)-\ell)_-}{R-r} \right|^{p(x,y)} \frac{\mathrm{d}y \,\mathrm{d}x}{|x-y|^{n-(1-s)p(x,y)}} \\
& \leq C |A_{\ell,R}^{-}| \left( \left|\frac{\ell}{R-r}\right|^{p_{+}} + \left|\frac{\ell}{R-r}\right|^{p_{-}}\right) \left( R^{(1-s)p_{+}} + R^{(1-s)p_{-}} \right) \\
&\leq C |A_{\ell,R}^{-}| \left(R^{-sp_{+}}\ell^{p_{+}} + R^{-sp_{-}}\ell^{p_{-}}\right),
\end{align*}
where we used \eqref{eq:P1} in the last inequality.
Next, we study $I_2$. Note that by the assumption $R^s\leq \delta H\leq \ell$ and \eqref{eq:P2}, we have 
\begin{equation}\label{eq:complementtoinside}
 R^{-s p_{\pm}(B_R\times B_R^c)} \ell^{p_{\pm}(B_R\times B_R^c)}   \leq R^{-s p_{\pm}} \ell^{p_{\pm}}.
\end{equation}
Using the nonnegativity of $u$ in $B_R$, the tail estimate \eqref{eq:assumption_tail2}, and \eqref{eq:complementtoinside}:
\begin{align*}
I_2&\leq C \sup_{x \in B_{\frac{R+r}{2}}} \int_{\R^n \setminus B_R} \frac{(u(y)-\ell)_-^{p(x,y)-1}}{|y-x_0|^{n+sp(x,y)}} \, \d y \int_{B_R} (u(x)-\ell)_- \,\mathrm{d}x\\
 & \leq C \sup_{x \in B_{\frac{R+r}{2}}}\left(\int_{\R^n \setminus B_R} \frac{u(y)_-^{p(x,y)-1}}{|y-x_0|^{n+sp(x,y)}} \, \d y 
+\int_{\R^n \setminus B_R} \left(\frac{\ell}{|y-x_0|^{s}}\right)^{p(x,y)} \frac{\ell^{-1}}{|y-x_0|^n} \, \d y\right) \ell |A_{\ell,R}^{-}|   \\
& \leq C\ell |A_{\ell,R}^{-}|\Bigg(\left( R^{-sp_+} (\delta H)^{p_+-1} + R^{-sp_-} (\delta H)^{p_--1} \right)\\
& \quad +\left( \ell^{p_{+}(B_R\times B_R^c)-1}R^{-sp_{+}(B_R\times B_R^c)} +\ell^{p_{-}(B_R\times B_R^c)-1}R^{-sp_{-}(B_R\times B_R^c)}  \right)\Bigg) \\
&\leq C|A_{\ell,R}^{-}|\left(R^{-sp_{+}}\ell^{p_{+}} + R^{-sp_{-}}\ell^{p_{-}}\right).
\end{align*}

Hence, since $R<1$,
\begin{equation} \label{eq:modular}
\begin{split}
&\varrho_{W^{s,p(\cdot,\cdot)}(B_{R/2})}((u-\ell)_-) + \int_{B_r} (u-\ell)_-(x) \int_{B_r} \frac{(u(y)-\ell)_+^{p(x,y)-1}}{|x-y|^{n+sp(x,y)}} \,\mathrm{d}y \,\mathrm{d}x \\
&\leq C |A_{\ell,R}^{-}| \left(R^{-sp_{+}}\ell^{p_{+}} + R^{-sp_{-}}\ell^{p_{-}}\right)\\
&\leq C |A_{\ell,R}^{-}| R^{-sp_{+}}\left(\ell^{p_{+}} + \ell^{p_{-}}\right).
\end{split}
\end{equation}
Combining the previous estimates, we get 
\begin{align*}
 & R^{\sigma q}[(u-\ell)_{-}]_{W^{\sigma,q}(B_{R/2})}^q \\
 & \leq C R^{\sigma q} \max \left\lbrace  |A_{\ell,r}^{-}|^{\frac{p_+-q}{p_+}} R^{q(s-\sigma)} ,  |A_{\ell,r}^{-}|^{\frac{p_--q}{p_-}} R^{(s-\sigma) \frac{p_+ q(p_--q)}{(p_+-q)p_-}} \right \rbrace \\
 & \quad \times \max\left\{ \left( |A_{\ell,R}^{-}| R^{-sp_{+}}\left(\ell^{p_{+}} + \ell^{p_{-}}\right)\right)^{q/p_{+}}, \left( |A_{\ell,R}^{-}| R^{-sp_{+}}\left(\ell^{p_{+}} + \ell^{p_{-}}\right)\right)^{q/p_{-}}\right\} \\
 & =:CR^{\sigma q}\max\{\Upsilon_1,\Upsilon_2\}\max\{\Phi_1,\Phi_2\}.
\end{align*}
We need to check the four possible cases for that inequality. Before doing that, 
note that since $\ell\in[2\delta H,  H]$ and $R^s\leq \delta H$, there is a constant $C>0$ such that
\begin{equation}\label{eq:logholdestell}
1+\ell^{p_{-}-p_{+}} \leq 1+R^{s(p_{-}-p_{+})} \leq 1+L^s \leq C,
\end{equation}
where we used \eqref{eq:P1}.
\begin{enumerate}[\text{Case} 1:]
\item We have by \eqref{eq:logholdestell}
\[ R^{\sigma q}\Upsilon_1 \Phi_1 =  |A_{\ell, R}^{-}| \left(\ell^{p_{+}}+\ell^{p_{-}}\right)^{\frac{q}{p_{+}}} \leq C |A_{\ell, R}^{-}| \ell^q.\]
\item By $\ell\in[2\delta H,  H]$ and the assumptions $H^{p_{+}-p_-}\leq 2$ and \eqref{eq:P1}, we get
\begin{align*}
 R^{\sigma q}\Upsilon_1 \Phi_2= |A_{\ell, R}^{-}|^{1+ \frac{q}{p_{-}} - \frac{q}{p_{+}}} R^{\frac{sq}{p_{-}}(p_{-}-p_{+})} \left(\ell^{p_{+}}+\ell^{p_{-}}\right)^{\frac{q}{p_{-}}} \leq C|A_{\ell, R}^{-}|^{1+ \frac{q}{p_{-}} - \frac{q}{p_{+}}} \ell^q.
 \end{align*}
\item 
Using \eqref{eq:logholdestell} together with \eqref{eq:P1}, we get
\begin{align*}
 R^{\sigma q}\Upsilon_2 \Phi_1 =  |A_{\ell, R}^{-}|^{1- \frac{q}{p_{-}} + \frac{q}{p_{+}}} R^{(s-\sigma)\frac{q^2(p_--p_+)}{(p_{+}-q)p_{-}}}  \left(\ell^{p_{+}}+\ell^{p_{-}}\right)^{\frac{q}{p_{+}}} \leq  C|A_{\ell, R}^{-}|^{1- \frac{q}{p_{-}} + \frac{q}{p_{+}}} \ell^q.
\end{align*} 
\item By $H^{p_{+}-p_-}\leq 2$, and \eqref{eq:P1}, we get
\begin{align*}
 R^{\sigma q}\Upsilon_2 \Phi_2 = |A_{\ell, R}^{-}| R^{\frac{(sp_{+}-\sigma q) q}{p_{-}(p_{+}-q)}(p_{-}-p_{+})} \left(\ell^{p_{+}} + \ell^{p_{-}}\right)^{\frac{q}{p_{-}}}  \leq C|A_{\ell, R}^{-}| \ell^q.
\end{align*} 
\end{enumerate}
Combining the estimates from the previous four cases, proves the assertion of the lemma.
\end{proof}

We are now in a position to prove the growth lemma. 
It is the main ingredient for the proof of the local H\"older regularity estimate.
 \begin{lemma} \label{lem:growth}
Let $B_R = B_R(x_0) \subset \mathbb{R}^n$ with $R \in (0,1)$. Let $H > 0$, and $0<\sigma < s <1$. Assume that $p$ satisfies \eqref{eq:P1} and \eqref{eq:P2} in $B_R$, $H^{p_+ - p_-} \leq 2$, and $p_+ < p_-^\ast$, where $p_\pm = p_\pm(B_R \times B_R)$ and $p_-^\ast = \frac{np_-}{n-\sigma p_-}$. Let $u \in W^{s, p(\cdot,\cdot)}(\mathbb{R}^n)$ be a weak supersolution to \eqref{eq:EL} in $B_R$ such that
\eqref{eq:assumption1} is satisfied for some $\gamma\in(0,1)$. Then there exists $\delta\in(0,1/8]$, such that, if $R^s\leq \delta H$ and \eqref{eq:assumption_tail2} is satisfied, then
\begin{equation}\label{eq:lemgrowth}
 u \geq \delta H \qquad \text{in } B_{R/4}.
 \end{equation}
 The constant $\delta$ depends on $n$, $s$, $\sigma$, $p_+(B_R \times \mathbb{R}^n)$, $p_-(B_R \times \mathbb{R}^n)$, $q$, and $L$.
\end{lemma}
\begin{proof}
The proof follows the ideas of \cite[Proof of Lemma 6.3]{Coz17}. 
Let $0<\delta<1/8$ and $0<\tau\leq 2^{-n-1}$ to be specified later. We first suppose 
\begin{equation}\label{eq:vol}
|B_{R/2} \cap \lbrace u < 2\delta H \rbrace| \leq \tau |B_{R/2}|
\end{equation}
and prove the assertion of the lemma under this additional assumption. Afterwards we prove that this precondition \eqref{eq:vol} is indeed a consequence of the given assumptions of the lemma.

We use $C>0$ for a constant depending on $n$, $s$, $\sigma$, $p_+(B_R \times \mathbb{R}^n)$, $p_-(B_R \times \mathbb{R}^n)$, $q$, and $L$ whose exact value is not important and that might change from line to line.

The idea to prove the assertion of the lemma is by iteration and the use of \Cref{lem:iteration}. For this purpose, we need to establish some auxiliary results. 
Let $\delta H\leq h < k \leq 2\delta H$ and $\frac{R}{4} \leq \rho < r \leq \frac{R}{2}$ . 
Note that by \eqref{eq:vol}, 
\begin{equation}\label{eq:vol12}
\begin{aligned}
|B_{\rho} \cap \{ (u-k)_{-} = 0 \} | &= |B_{\rho}\setminus \{u<k\}| \geq |B_{\rho}| - |B_{R/2} \cap\{u<2\delta H\}| \\
& \geq |B_{\rho}| - \tau |B_{R/2}|  = \left(1-\tau \left(\frac{R/2}{\rho}\right)^n\right)|B_{\rho}| \geq \left( 1-2^n\tau \right)|B_{\rho}| \\
& \geq \frac{1}{2}|B_{\rho}|.
\end{aligned}
\end{equation}
Using \eqref{eq:vol12}, \Cref{thm:homo_sobolev} and \Cref{lemma:ellest}, we have
\begin{equation}
\begin{aligned} \label{eq:tointerate}
(k-h) |A_{h,\rho}^{-}|^{\frac{n-\sigma}{n}}  &\leq \left(\int_{A_{h,\rho}^{-}} (k-u(x))^{\frac{n}{n-\sigma}} \, \d x  \right)^{\frac{n-\sigma}{n}} \\
& \leq \left(\int_{B_{\rho}} (u(x)-k)_{-}^{\frac{n}{n-\sigma}} \, \d x  \right)^{\frac{n-\sigma}{n}} \\
& \leq C\int_{B_{\rho}}\int_{B_{\rho}} \frac{|(u(x)-k)_{-} - (u(y)-k)_{-}|}{|x-y|^{n+\sigma}} \, \d x\, \d y \\
&\leq C k r^{-\sigma} \max \left\lbrace |A_{k, r}^-|, |A_{k, r}^-|^{1+\frac{1}{p_-}-\frac{1}{p_+}}, |A_{k,r}^-|^{1+\frac{1}{p_+} - \frac{1}{p_-}} \right\rbrace,
\end{aligned}
\end{equation}
where $A_{k, r}^- = B_r \cap \lbrace u < r \rbrace$. In the proceeding, we use \eqref{eq:tointerate} to prove the assertion of the lemma by iteration.
We define for $j\in \N\cup\{0\}$
\[ r_j = \frac{1}{4}(1+2^{-j})R, \quad k_j=(1+2^{-j})\delta H, \quad \text{and} \, \,   y_j = \frac{|A_{k_{j}, r_{j}}^{-}|}{|B_{r_j}|}. \]
Then $r_j \in (\frac{1}{4}R, \frac{1}{2}R]$ and $k_j\in (\delta H, 2\delta H]$. Choosing $k=k_{j}$, $h=k_{j+1}$, $\rho=r_{j+1}$, and $r=r_{j}$, we get from \eqref{eq:tointerate}
\[
\frac{\delta H}{2^{j+1}} \left(y_{j+1}|B_{r_{j+1}}|\right)^{\frac{n-\sigma}{n}} 
\leq C (\delta H) r_j^{-\sigma}\max\left\{ r_j^n y_j, (r_j^{n}y_j)^{1+\frac{p_{+}-p_{-}}{p_{+}p_{-}}}, (r_j^{n}y_j)^{1+\frac{p_{-}-p_{+}}{p_{+}p_{-}}} \right\},
\]
which leads to
\begin{equation}\label{eq:yj+1yj}
y_{j+1} \leq C \left( 2^j r_j^{-n}\max\left\{ r_j^n y_j, (r_j^{n}y_j)^{1+\frac{p_{+}-p_{-}}{p_{+}p_{-}}}, (r_j^{n}y_j)^{1+\frac{p_{-}-p_{+}}{p_{+}p_{-}}} \right\}\right)^{\frac{n}{n-\sigma}}. 
\end{equation} 
If we prove that there are $\beta_1, \beta_2,\beta_3 >0$ such that 
\begin{equation}\label{eq:aimineqbeta}
 y_{j+1} \leq C 2^{\frac{n}{n-\sigma}j} \max\left\{ y_j^{1+\beta_1},y_j^{1+\beta_2},y_j^{1+\beta_3}\right\}
\end{equation}
and $y_0$ is sufficiently small, then we can apply \Cref{lem:iteration} which would prove \eqref{eq:lemgrowth}.
We have three cases for the maximum in \eqref{eq:yj+1yj}:
\begin{enumerate}[\text{Case} 1:]
\item In the first case, we have
\[ y_{j+1} \leq C2^{\frac{n}{n-\sigma}j} y_{j}^{\frac{n}{n-\sigma}}. \]
Since $\frac{n}{n-\sigma}>1$, this proves the assertion in the first case.
\item In the second case, using $r_j\leq 1$ and the fact that its exponent is positive,
\[ y_{j+1} \leq C2^{\frac{n}{n-\sigma}j} r_j^{\frac{p_{+}-p_{-}}{p_{+}p_{-}}\frac{n}{n-\sigma}} y_{j}^{\frac{n}{n-\sigma} + \frac{p_{+}-p_{-}}{p_{+}p_{-}}\frac{n}{n-\sigma}} \leq C2^{\frac{n}{n-\sigma}j} y_{j}^{\frac{n}{n-\sigma} + \frac{p_{+}-p_{-}}{p_{+}p_{-}}\frac{n}{n-\sigma}}. \]
Since ${\frac{n}{n-\sigma} + \frac{p_{+}-p_{-}}{p_{+}p_{-}}\frac{n}{n-\sigma}} > {\frac{n}{n-\sigma}} >1$, we have proven the assertion in the second case.
\item In the third case, using \eqref{eq:P1}, we have
\[ y_{j+1} \leq C2^{\frac{n}{n-\sigma}j} r_j^{\frac{n}{(n-\sigma)p_{+}p_{-}}(p_{-}-p_{+})} y_{j}^{\frac{n}{n-\sigma}\left( 1 + \frac{1}{p_{+}} - \frac{1}{p_{-}}\right)} \leq  C2^{\frac{n}{n-\sigma}j} y_{j}^{\frac{n}{n-\sigma}\left( 1 + \frac{1}{p_{+}} - \frac{1}{p_{-}}\right)}.  \]
Note that by assumption $p_+ < p_-^\ast$, where $p_-^\ast = \frac{np_-}{n-\sigma p_-}$, which is equivalent to
\[ p_+ < p_-^\ast \iff \frac{\sigma}{n}>\frac{1}{p_{-}} - \frac{1}{p_{+}} \iff \frac{n}{n-\sigma}\left( 1 + \frac{1}{p_{+}} - \frac{1}{p_{-}}\right) >1. \]
This completes the proof in this case.
\end{enumerate}
Hence, we have shown \eqref{eq:aimineqbeta}. Note that by \eqref{eq:vol}
\[ y_0 = \frac{|A_{2\delta H, \frac{R}{2}}^{-}|}{|B_{\frac{R}{2}}|}\leq \tau. \]
Choosing $\tau$ sufficiently small, allows us to apply \Cref{lem:iteration} which yields $y_j\to 0$ as $j\to\infty$ and proves
\eqref{eq:lemgrowth}.

In the remainder of the proof, we show \eqref{eq:vol}. We prove this assertion by contradiction.
Hence, suppose that \eqref{eq:vol} is not true, i.e.
\begin{equation}\label{eq:contrad}
|B_{R/2} \cap \lbrace u < 2\delta H \rbrace| > \tau |B_{R/2}|.
\end{equation} 
We split the proof into two cases for $s\in(0,1)$. First, when $s$ is sufficiently large, we prove the assertion using an isoperimetric-type inequality by Cozzi \cite[Proposition 5.1]{Coz17}. Second, in the case of small $s$, the assertion follows by a direct calculation.

Let $\bar{s}$ be the constant coming from the isoperimetric-type inequality \cite[Proposition 5.1]{Coz17} (applied for the constant exponent case $q$ and for $\sigma$) and let $s\in[\bar{s},1)$. For given $\delta$, there is a unique $m\in\N$ such that
\[ 2^{-m-1}\leq \delta < 2^{-m}. \]
We define for $i = 0,\dots,m-1$, $k_i= 2^{-i}H.$ Note that by definition $k_i\in (2\delta H, H]$.
In the following, we check the conditions to apply \cite[Proposition 5.1]{Coz17}. By \eqref{eq:assumption1} and \eqref{eq:contrad}, we get
\[ |B_{R/2} \cap \lbrace (u-k_{i-1})_{-} \leq 2^{-i}H \rbrace| = |B_{R/2} \cap \lbrace u \geq k_i \rbrace| \geq |B_{R/2} \cap \lbrace u \geq H \rbrace|  \geq \gamma |B_{R/2}|   \]
and
\[ |B_{R/2} \cap \lbrace (u-k_{i-1})_{-} \geq 3\cdot 2^{-i-1}H \rbrace| = |B_{R/2} \cap \lbrace u \leq k_{i+1} \rbrace| \geq |B_{R/2} \cap \lbrace u < 2\delta H \rbrace|  \geq \tau |B_{R/2}|  \]
for $i=1,\dots, m-2$.
In order to apply \cite[Proposition 5.1]{Coz17}, it remains to prove that there is a constant $C>0$ such that
\begin{equation}\label{eq:condisoper}
\|(u-k_{i-1})_{-}\|_{L^q(B_{R/2})}^q + R^{\sigma q} [(u-k_{i-1})_{-}]_{W^{\sigma,q}(B_{R/2})}^q \leq C(k_i-k_{i+1})^q R^n.
\end{equation} 
Using the nonnegativity of $u$ in $B_{R}$, we get 
\[ \|(u-k_{i-1})_{-}\|_{L^q(B_{R/2})}^q \leq C k_{i-1}^q R^n. \]
Combining this estimate together with \Cref{lemma:ellest} for $\ell=k_{i-1}$,
\begin{equation} \label{eq:k-k}
\begin{split}
&\|(u-k_{i-1})_{-}\|_{L^q(B_{R/2})}^q +  R^{\sigma q}[(u-k_{i-1})_{-}]_{W^{\sigma,q}(B_{R/2})}^q \\
&\leq Ck_{i-1}^q \max \left\lbrace |A_{k_{i-1}, R}^-|, |A_{k_{i-1}, R}^-|^{1+\frac{q}{p_-}-\frac{q}{p_+}}, |A_{k_{i-1}, R}^-|^{1+\frac{q}{p_+} - \frac{q}{p_-}} \right\rbrace \\
&\leq C(k_{i} - k_{i+1})^q R^n
\end{split}
\end{equation}
for some constant $C=C(n, s, \sigma, p_+(B_R \times \mathbb{R}^n), p_-(B_R \times \mathbb{R}^n), q, L) > 0$, where we used \eqref{eq:P1} in the last inequality. This proves \eqref{eq:condisoper} and therefore, we can apply \cite[Proposition 5.1]{Coz17} with $h=k_{i-1}-k_i$, $k=k_{i-1}-k_{i+1}$, and the function $(u-k_{i-1})_-$, that yields
\begin{equation}\label{eq:aplProp5.1Coz}
\begin{aligned} 
(k_{i}-k_{i+1})& \left[ |B_{R/2}\cap \{ u \geq k_i\}||B_{R/2}\cap \{ u \leq k_{i+1}\}| \right]^{\frac{n-1}{n}} \\
&\leq CR^{n-2+\sigma} [(u-k_{i-1})_{-}]_{W^{\sigma,q}(B_{R/2})} \left|B_{R/2}\cap \{k_{i+1} < u \leq k_i\}\right|^{\frac{q-1}{q}}.
\end{aligned}
\end{equation}
In the following, we show that this inequality leads to a contradiction.
On the one hand, the left-hand side can be estimated with \eqref{eq:assumption1} by
\[ (k_{i}-k_{i+1}) \left[ |B_{R/2}\cap \{ u \geq k_i\}||B_{R/2}\cap \{ u \leq k_{i+1}\}| \right]^{\frac{n-1}{n}} \geq Ck_{i+1}\left[R^n|B_{R/2}\cap \{ u < 2\delta H\}|\right]^{\frac{n-1}{n}}.\]
On the other hand, we can estimate the right-hand side, using \eqref{eq:k-k}, by
\begin{align*}
 R^{n-2+\sigma} [(u-k_{i-1})_{-}]_{W^{\sigma,q}(B_{R/2})}& \left|B_{R/2}\cap \{k_{i+1} < u \leq k_i\}\right|^{\frac{q-1}{q}} \\
 & \leq  CR^{n-2+\frac{n}{q}} k_{i+1} \left|B_{R/2}\cap \{k_{i+1} < u \leq k_i\}\right|^{\frac{q-1}{q}}.
 \end{align*}
Hence, we get from \eqref{eq:aplProp5.1Coz}
\[ |B_{R/2}\cap \{ u < 2\delta H\}|^{\frac{q(n-1)}{(q-1)n}}\leq CR^{\frac{n-q}{q-1}}\left|B_{R/2}\cap \{k_{i+1} < u \leq k_i\}\right|. \]
Summing up this inequality over $i=1,\dots,m-2$ gives us
\[ (m-2)\left[|B_{R/2}\cap \{ u < 2\delta H\}|\right]^{\frac{q(n-1)}{(q-1)n}}\leq CR^{\frac{n-q}{q-1}}\left|B_{R/2}\right| = CR^{\frac{q(n-1)}{q-1}}, \]
which leads to
\[ |B_{R/2}\cap \{ u < 2\delta H\}|\leq CR^{n}m^{-\frac{n(q-1)}{(n-1)q}} \leq C|B_{R/2}| |\log \delta|^{-\frac{n(q-1)}{(n-1)q}}. \]
Estimating the left-hand side by \eqref{eq:contrad}, we get
\[ |\log \delta|^{-\frac{n(q-1)}{(n-1)q}} \geq C. \]
Hence, choosing $\delta$ sufficiently small results in a contradiction and finishes the proof for the case $s\in [\bar{s},1)$. 

Now let $s\in (0,\bar{s})$. In this case, we get by \eqref{eq:modular}, \eqref{eq:assumption1}, and \eqref{eq:contrad}
\begin{align*}
&\left((4\delta H)^{p_+} + (4\delta H)^{p_{-}}\right) R^{n-sp_{+}} \geq C\int_{B_{R/2}}\int_{B_{R/2}} \frac{(u(x)-4\delta H)_{+}^{p(x,y)-1} (u(y)-4\delta H)_{-}}{|x-y|^{n+sp(x,y)}}\, \d y\, \d x \\
& \geq \frac{C}{R^{n+sp_{-}}} \int_{B_{R/2}\cap \{u\geq H\}} (u(x)-4\delta H)^{p(x,y)-1} \, \d x \int_{B_{R/2}\cap \{u<2\delta H\}} (4\delta H-u(y)) \d y \\
& \geq \frac{C}{R^{n+sp_{-}}} |B_{R/2}\cap \{u\geq H\}| \min\left\lbrace \left( \frac{H}{2}\right)^{p_+-1}, \left( \frac{H}{2}\right)^{p_--1} \right\rbrace 2\delta H |B_{R/2}\cap \{u<2\delta H\}| \\
& \geq C R^{n-sp_{-}} \delta \min\left\{ H^{p_+}, H^{p_-}\right\}.
\end{align*}
Hence, since by $R^{sp_{-}-sp_{+}}\leq L^s$ by \eqref{eq:P1}, we get
\[ \left((4\delta H)^{p_+} + (4\delta H)^{p_{-}}\right) \geq C \delta \min\left\{ H^{p_+}, H^{p_-}\right\}.\]
Choosing $\delta$ sufficiently small leads to a contradiction in this inequality and finishes the proof of the lemma.
\end{proof}
We would like to emphasize that we made use of \Cref{lem:embedding} in the foregoing proof. For this reason we were able to prove the growth lemma without using the Sobolev inequality for variable exponents. It was sufficient to make use the fractional Sobolev inequality for constant exponents. 

\begin{proof} [Proof of \Cref{thm:Holder}]
Let $x_0\in\R^n$. If $p(x_0, x_0) > n/s$, then we can find $R > 0$ and $\alpha \in (0,1)$ such that $B_R(x_0) \Subset \Omega$ and $u \in C^\alpha(\overline{B_{R}(x_0)})$ as in the proof of \Cref{thm:local_boundedness}. Thus, let us assume $p(x_0, x_0) \leq n/s$ in the rest of the proof. In this case, for given $\sigma \in (0,s)$ we can find $R \in (0,1)$ such that $B_R(x_0) \Subset \Omega$ and $p_+(B_R(x_0)\times B_R(x_0)) < p_-^{\ast}(B_R(x_0)\times B_R(x_0))$, where $p_-^{\ast}(B_R(x_0)\times B_R(x_0))= \frac{np_-(B_R(x_0)\times B_R(x_0))}{n-\sigma p_-(B_R(x_0)\times B_R(x_0))}$. By \Cref{thm:local_boundedness}, $u\in L^{\infty}(B_R(x_0))$.

Let $\delta\in(0,1)$ be the constant from \Cref{lem:growth} and let 
\begin{equation}\label{eq:alphabounds}
0<\alpha < \min\left\{ s,\log_4\left(\frac{2}{2-\delta}\right), \frac{sp_{+}(\Omega\times\R^n)}{2(p_{+}(\Omega\times\R^n)-1)}\right\}
\end{equation}
be chosen such that the following is satisfied:
\begin{equation}\label{eq:propalpha1}
\int^{\infty}_1 \frac{((4t)^{\alpha} - 1)^{p_{+}(\Omega\times \R^n)-1}}{t^{1+sp_{+}(\Omega\times\R^n)}}\, \d t + \int^{\infty}_1 \frac{((4t)^{\alpha} - 1)^{p_{-}(\Omega\times \R^n)-1}}{t^{1+sp_{-}(\Omega\times\R^n)}}\, \d t \leq \frac{\delta^{p_{+}(\Omega\times\R^n)-1}}{2^{p_{+}(\Omega\times\R^n)}n\omega_n},
\end{equation}
where $\omega_n$ denotes the volume the $n$-dimensional Euclidean unit ball.
We define $j_0\in\N$ to be the smallest natural number satisfying
\begin{equation}\label{eq:j0}
\begin{aligned}
j_0\geq \max\Bigg\{ &\frac{sp_{+}(\lbrace x_0 \rbrace \times B_R^c)}{2}\left|\log_4\left(\frac{\delta^{p_{+}(\lbrace x_0 \rbrace \times B_R^c)-1}}{2C_0}\right) \right|, \\
&\frac{sp_{-}(\Omega \times \R^n)}{2}\left|\log_4\left(\frac{\delta^{p_{-}(\Omega \times \R^n)-1}}{2C_0}\right) \Bigg|, \frac{|\log_4 (\frac{\delta}{2})|}{s-\alpha} \right\},
\end{aligned}
\end{equation}
where $C_0:=\max\{1,2^{p_{+}(\Omega\times\R^n)}\}\left(\frac{n\omega_n}{sp_{-}(\Omega\times\R^n)}+1\right)$.

In the following we show that there is a non-increasing sequence $(M_j)$ and a non-decreasing sequence $(m_j)$ in $\R$, such that for all $j\in\N\cup\{0\}$ 
\begin{equation}\label{eq:sequencesMjmj}
m_j \leq u \leq M_j \quad \text{in } B_{4^{-j}R}(x_0) \qquad \text{and} \qquad M_j - m_j = Z4^{-\alpha j},
\end{equation}
where 
\begin{align*}
Z&:= 2\cdot 4^{\alpha j_0} \|u\|_{L^{\infty}(B_R(x_0))} + R^s + 1 \\
& \quad + \left(R^{sp_{+}(\{x_0\}\times B_R(x_0)^c)} \sup_{x \in B_{3R/4}(x_0)} \int_{\mathbb{R}^n \setminus B_R(x_0)} \frac{|u(y)|^{p(x,y)-1}}{|y-x_0|^{n+sp(x,y)}} \,\mathrm{d}y\right)^{\frac{1}{p_{+}(\{x_0\}\times B_R(x_0)^c)-1}} . 
\end{align*}
For $j\in \{0,\dots,j_0\}$, we define $M_j:=4^{-\alpha j} Z/2$ and $m_j:=-4^{-\alpha j} Z/2$. Then \eqref{eq:sequencesMjmj} is clearly satisfied for all $j\in \{0,\dots,j_0\}$. It remains to prove the assertion \eqref{eq:sequencesMjmj} for $j > j_0$. 
The proof of the assertion follows by induction. Let us fix $j\geq j_0$ and assume that \eqref{eq:sequencesMjmj} is true for all $i\in \{0,\dots,j\}$. We now construct the elements $m_{j+1}$ and $M_{j+1}$ of the sequences. We distinguish between two cases.  

First, we assume 
\begin{equation}\label{eq:Hoeldcase1}
 \left| B_{\frac{4^{-j}R}{2}}(x_0) \cap \left\{ u\geq m_{j} + \frac{M_j-m_j}{2} \right\} \right| \geq \frac{1}{2}\left| B_{\frac{4^{-j}R}{2}}(x_0) \right|. 
\end{equation}
In this case, we define $v:=u-m_j$, $H:=\frac{M_j-m_j}{2}$ and $\widetilde{R}:=4^{-j}R$. \\
The main idea for constructing $m_{j+1}$ and $M_{j+1}$ is to apply \Cref{lem:growth} for the function $v$ and the radius $\widetilde{R}$. Hence, we need to verify the requirements of the lemma.
Note that by assumption we have $0\leq v\leq 2H$ in $B_{\widetilde{R}}(x_0)$ and  $| B_{\widetilde{R}/2}(x_0) \cap \left\{ v\geq H \right\} | \geq \frac{1}{2}| B_{\widetilde{R}/2}(x_0)|.$ It remains to prove $\widetilde{R}^s\leq \delta H$ and \eqref{eq:assumption_tail2}. First we show $\widetilde{R}^s\leq \delta H$. Note that $2H = M_j - m_j = Z4^{-\alpha j} \geq R^s 4^{-\alpha j}$. Since $j\geq j_0$, we can use \eqref{eq:j0}, which leads to
\[ \widetilde{R}^s = 4^{-js} R^s \leq  4^{j(\alpha-s)}2H \leq \delta H.\]
It remains to prove \eqref{eq:assumption_tail2}. We split $\R^n\setminus B_{\widetilde{R}}(x_0)$ as follows:
\[ \R^n\setminus B_{\widetilde{R}}(x_0) = \left(\R^n\setminus B_R(x_0)\right) \cup \left(\bigcup_{l=0}^{j-1} B_{4^{-l}R}(x_0) \setminus B_{4^{-(l+1)}R}(x_0)\right). \]
If $x\in B_{4^{-l}R}(x_0) \setminus B_{4^{-(l+1)}R}(x_0)$, then $|x-x_0|\geq 4^{-l-1}R$ and therefore
\begin{align*}
v(x) &= u(x)-m_j \geq m_l-M_l + 2H  = 2H(-4^{(-l+j)\alpha}+1) \\
& \geq -2H\left(\left( \frac{4|x-x_0|}{\widetilde{R}} \right)^{\alpha}-1\right). 
\end{align*}
On the other hand, if $x\in \R^n\setminus B_R(x_0)$, we have $v(x) \geq -|u(x)|-Z/2$.
 Now we are in a position to finalize the verification of \eqref{eq:assumption_tail2}. By the previous estimates on $v$ in $\R^n\setminus B_{\widetilde{R}}(x_0)$, we have
\begin{align*}
&\sup_{x \in B_{3\widetilde{R}/4}(x_0)} \int_{\mathbb{R}^n \setminus B_{\widetilde{R}}(x_0)} \frac{v_-(y)^{p(x,y)-1}}{|y-x_0|^{n+sp(x,y)}} \,\mathrm{d}y \\
&\leq \sup_{x \in B_{3\widetilde{R}/4}(x_0)} \int_{\R^n \setminus B_{\widetilde{R}}(x_0)} \frac{\left(2H\left(\left( \frac{4|y-x_0|}{\widetilde{R}} \right)^{\alpha}-1\right)\right)^{p(x,y)-1} }{|y-x_0|^{n+sp(x,y)}} \,\mathrm{d}y  \\
&\qquad +\max\{1,2^{p_+(B_R(x_0)\times\R^n)-1}\}\sup_{x \in B_{3\widetilde{R}/4}(x_0)} \int_{\mathbb{R}^n \setminus B_{R}(x_0)} \frac{|u(y)|^{p(x,y)-1}+Z^{p(x,y)-1}}{|y-x_0|^{n+sp(x,y)}} \,\mathrm{d}y\\
&  =:J_1+J_2.
\end{align*}
First note, that we can estimate $J_1$ as follows:
\begin{align*}
J_1&\leq \int_{\mathbb{R}^n \setminus B_{\widetilde{R}}(x_0)} \,\frac{\left(2H\left(\left( \frac{4|y-x_0|}{\widetilde{R}} \right)^{\alpha}-1\right)\right)^{p_+(B_{\widetilde{R}}(x_0) \times B_{\widetilde{R}}(x_0)^c)-1}}{|y-x_0|^{n+sp_+(B_{\widetilde{R}}(x_0) \times B_{\widetilde{R}}(x_0)^c)}} \, \d y\\
&\quad + \int_{\mathbb{R}^n \setminus B_{\widetilde{R}}(x_0)} \frac{\left(2H\left(\left( \frac{4|y-x_0|}{\widetilde{R}} \right)^{\alpha}-1\right)\right)^{p_-(B_{\widetilde{R}}(x_0) \times B_{\widetilde{R}}(x_0)^c)-1}}{|y-x_0|^{n+sp_-(B_{\widetilde{R}}(x_0) \times B_{\widetilde{R}}(x_0)^c)}}\, \d y\\
& = n\omega_n (2H)^{p_+(B_{\widetilde{R}}(x_0) \times B_{\widetilde{R}}(x_0)^c)-1} \widetilde{R}^{-sp_+(B_{\widetilde{R}}(x_0) \times B_{\widetilde{R}}(x_0)^c)}\int^{\infty}_1 \frac{((4t)^{\alpha} - 1)^{p_{+}(B_{\widetilde{R}}(x_0)\times B_{\widetilde{R}}(x_0)^c)-1}}{t^{1+sp_{+}(B_{\widetilde{R}}(x_0)\times B_{\widetilde{R}}(x_0)^c)}}\, \d t \\
& \quad + n\omega_n (2H)^{p_-(B_{\widetilde{R}}(x_0) \times B_{\widetilde{R}}(x_0)^c)-1} \widetilde{R}^{-sp_-(B_{\widetilde{R}}(x_0) \times B_{\widetilde{R}}(x_0)^c)}\int^{\infty}_1 \frac{((4t)^{\alpha} - 1)^{p_{-}(B_{\widetilde{R}}(x_0)\times B_{\widetilde{R}}(x_0)^c)-1}}{t^{1+sp_{-}(B_{\widetilde{R}}(x_0)\times B_{\widetilde{R}}(x_0)^c)}}\, \d t.
\end{align*}
Using $p_{-}(\Omega\times\R^n)\leq p_{\pm}(B_{\widetilde{R}}(x_0)\times B_{\widetilde{R}}(x_0)^c) \leq p_{+}(\Omega\times\R^n)$ and \eqref{eq:propalpha1}, we get 
\begin{align*}
\int^{\infty}_1& \frac{((4t)^{\alpha} - 1)^{p_{\pm}(B_{\widetilde{R}}(x_0)\times B_{\widetilde{R}}(x_0)^c)-1}}{t^{1+sp_{\pm}(B_{\widetilde{R}}(x_0)\times B_{\widetilde{R}}(x_0)^c)}}\, \d t \\
&\leq \int^{\infty}_1 \frac{((4t)^{\alpha} - 1)^{p_{+}(\Omega\times \R^n)-1}}{t^{1+sp_{+}(\Omega\times \R^n)}}\, \d t
+ \int^{\infty}_1 \frac{((4t)^{\alpha} - 1)^{p_{-}(\Omega\times \R^n)-1}}{t^{1+sp_{-}(\Omega\times \R^n)}}\, \d t \\
& \leq \frac{\delta^{p_{+}(\Omega\times\R^n)-1}}{2^{p_{+}(\Omega\times\R^n)}n\omega_n}.
\end{align*}
Combing the previous two estimates, we arrive at 
\begin{align*}
J_1 & \leq  \frac12 H^{p_+(B_{\widetilde{R}}(x_0) \times B_{\widetilde{R}}(x_0)^c)-1} \widetilde{R}^{-sp_+(B_{\widetilde{R}}(x_0) \times B_{\widetilde{R}}(x_0)^c)} \delta^{p_{+}(\Omega\times\R^n)-1} \\
& \quad +\frac12 H^{p_-(B_{\widetilde{R}}(x_0) \times B_{\widetilde{R}}(x_0)^c)-1} \widetilde{R}^{-sp_-(B_{\widetilde{R}}(x_0) \times B_{\widetilde{R}}(x_0)^c)}\delta^{p_{+}(\Omega\times\R^n)-1} \\
& \leq \frac12 (\delta H)^{p_+(B_{\widetilde{R}}(x_0) \times B_{\widetilde{R}}(x_0)^c)-1} \widetilde{R}^{-sp_+(B_{\widetilde{R}}(x_0) \times B_{\widetilde{R}}(x_0)^c)}  \\
& \quad +\frac12 (\delta H)^{p_-(B_{\widetilde{R}}(x_0) \times B_{\widetilde{R}}(x_0)^c)-1} \widetilde{R}^{-sp_-(B_{\widetilde{R}}(x_0) \times B_{\widetilde{R}}(x_0)^c)} \\
& \leq \frac12 (\delta H)^{p_+(B_{\widetilde{R}}(x_0) \times B_{\widetilde{R}}(x_0))-1} \widetilde{R}^{-sp_+(B_{\widetilde{R}}(x_0) \times B_{\widetilde{R}}(x_0))}  \\
& \quad +\frac12 (\delta H)^{p_-(B_{\widetilde{R}}(x_0) \times B_{\widetilde{R}}(x_0))-1} \widetilde{R}^{-sp_-(B_{\widetilde{R}}(x_0) \times B_{\widetilde{R}}(x_0))}.
\end{align*}
In the last inequality we used that by \eqref{eq:P2} and $\widetilde{R}^s\leq \delta H$,
\[  R^{-s p_{\pm}(B_R(x_0)\times B_R(x_0)^c)} (\delta H)^{p_{\pm}(B_R(x_0)\times B_R(x_0)^c)}   \leq R^{-s p_{\pm}(B_R(x_0)\times B_R(x_0))} (\delta H)^{p_{\pm}(B_R(x_0)\times B_R(x_0))}. \]
Next, we estimate $J_2$ as follows:
\begin{equation}\label{eq:estJ2}
\begin{aligned}
J_2 &\leq \max\{1,2^{p_+(\Omega\times\R^n)-1}\} R^{-sp_{+}(\{x_0\}\times B_R(x_0)^c)}Z^{p_{+}(\{x_0\}\times B_R(x_0)^c)-1} \\
&\quad + \max\{1,2^{p_+(\Omega\times\R^n)-1}\}  \\
&\qquad \times \int_{\mathbb{R}^n \setminus B_{R}(x_0)} \left(\frac{Z^{p_{+}(B_{\widetilde{R}}(x_0)\times B_{\widetilde{R}}(x_0)^c)-1}}{|y-x_0|^{n+sp_{+}(B_{\widetilde{R}}(x_0)\times B_{\widetilde{R}}(x_0)^c)}}  + \frac{Z^{p_{-}(B_{\widetilde{R}}(x_0)\times B_{\widetilde{R}}(x_0)^c)-1}}{|y-x_0|^{n+sp_{-}(B_{\widetilde{R}}(x_0)\times B_{\widetilde{R}}(x_0)^c)}} \right)\,\mathrm{d}y \\
& \leq C_0R^{-sp_{+}(B_{\widetilde{R}}(x_0)\times B_{\widetilde{R}}(x_0)^c)}Z^{p_{+}(B_{\widetilde{R}}(x_0)\times B_{\widetilde{R}}(x_0)^c)-1} \\
& \quad + C_0R^{-sp_{-}(B_{\widetilde{R}}(x_0)\times B_{\widetilde{R}}(x_0)^c)}Z^{p_{-}(B_{\widetilde{R}}(x_0)\times B_{\widetilde{R}}(x_0)^c)-1} \\
& =C_0(4^{j}\widetilde{R})^{-sp_{+}(B_{\widetilde{R}}(x_0)\times B_{\widetilde{R}}(x_0)^c)}(4^{\alpha j}2H)^{p_{+}(B_{\widetilde{R}}(x_0)\times B_{\widetilde{R}}(x_0)^c)-1} \\
& \quad + C_0(4^{j}\widetilde{R})^{-sp_{-}(B_{\widetilde{R}}(x_0)\times B_{\widetilde{R}}(x_0)^c)}(4^{\alpha j}2H)^{p_{-}(B_{\widetilde{R}}(x_0)\times B_{\widetilde{R}}(x_0)^c)-1} \\
& \leq C_04^{\frac{-sp_{+}(B_{\widetilde{R}}(x_0)\times B_{\widetilde{R}}(x_0)^c)}{2}j_0}H^{p_{+}(B_{\widetilde{R}}(x_0)\times B_{\widetilde{R}}(x_0)^c)-1}\widetilde{R}^{-sp_{+}(B_{\widetilde{R}}(x_0)\times B_{\widetilde{R}}(x_0)^c)} \\
& \quad +C_04^{\frac{-sp_{-}(B_{\widetilde{R}}(x_0)\times B_{\widetilde{R}}(x_0)^c)}{2}j_0}H^{p_{-}(B_{\widetilde{R}}(x_0)\times B_{\widetilde{R}}(x_0)^c)-1}\widetilde{R}^{-sp_{-}(B_{\widetilde{R}}(x_0)\times B_{\widetilde{R}}(x_0)^c)} \\
& \leq \frac{1}{2}(\delta H)^{p_+(B_{\widetilde{R}}(x_0)\times B_{\widetilde{R}}(x_0)^c) - 1} \widetilde{R}^{-sp_+(B_{\widetilde{R}}(x_0)\times B_{\widetilde{R}}(x_0)^c)}\\
&\quad + \frac{1}{2}(\delta H)^{p_-(B_{\widetilde{R}}(x_0)\times B_{\widetilde{R}}(x_0)^c)-1} \widetilde{R}^{-sp_-(B_{\widetilde{R}}(x_0)\times B_{\widetilde{R}}(x_0)^c)} \\
& \leq \frac{1}{2}(\delta H)^{p_+(B_{\widetilde{R}}(x_0)\times B_{\widetilde{R}}(x_0)) - 1} \widetilde{R}^{-sp_+(B_{\widetilde{R}}(x_0)\times B_{\widetilde{R}}(x_0))} \\
&\quad +\frac{1}{2}(\delta H)^{p_-(B_{\widetilde{R}}(x_0)\times B_{\widetilde{R}}(x_0))-1} \widetilde{R}^{-sp_-(B_{\widetilde{R}}(x_0)\times B_{\widetilde{R}}(x_0))},
\end{aligned}
\end{equation}
where we used the definition of $Z$, \eqref{eq:P2}, \eqref{eq:alphabounds} and \eqref{eq:j0}. 
Note that the constant $C_0$ comes from \eqref{eq:j0}. Combining the estimates of $J_1$ and $J_2$, proves \eqref{eq:assumption_tail2}.

Hence, we can apply \Cref{lem:growth}, which leads to
\[ u\geq m_j+\delta H = m_j+\delta \frac{M_{j}-m_j}{2} = m_j + \frac{\delta 4^{-\alpha j}Z}{2} > m_j+4^{-\alpha j}(1-4^{-\alpha})Z \quad \text{in } B_{\widetilde{R}/4}(x_0), \]
where we used \eqref{eq:alphabounds} in the last inequality. 
Hence, choosing $M_{j+1} = M_j$ and $m_{j+1}=m_j+4^{-\alpha j}(1-4^{-\alpha})Z$ proves \eqref{eq:sequencesMjmj} for the case \eqref{eq:Hoeldcase1}.

In the second case
\[ \left| B_{\frac{4^{-j}R}{2}}(x_0) \cap \left\{ u\geq m_{j} + \frac{M_j-m_j}{2} \right\} \right| < \frac{1}{2}\left| B_{\frac{4^{-j}R}{2}}(x_0) \right|, \]
we can proceed similarly and consider the function $v:=M_j-u$. In this case, we can choose the members of the sequences to be of the form $M_{j+1} = M_j-4^{-\alpha j}(1-4^{-\alpha})Z$ and $m_{j+1}=m_j$. This completes the construction of the sequences $(M_j)$ and $(m_j)$ and completes the proof of \eqref{eq:sequencesMjmj}. Now the local Hölder regularity follows in a standard way.
\end{proof}


\begin{thebibliography}{10}

\bibitem{ABF03}
E.~Acerbi, G.~Bouchitt\'{e}, and I.~Fonseca.
\newblock Relaxation of convex functionals: the gap problem.
\newblock {\em Ann. Inst. H. Poincar\'{e} Anal. Non Lin\'{e}aire},
  20(3):359--390, 2003.

\bibitem{AF94}
E.~Acerbi and N.~Fusco.
\newblock A transmission problem in the calculus of variations.
\newblock {\em Calc. Var. Partial Differential Equations}, 2(1):1--16, 1994.

\bibitem{AM01}
E.~Acerbi and G.~Mingione.
\newblock Regularity results for a class of functionals with non-standard
  growth.
\newblock {\em Arch. Ration. Mech. Anal.}, 156(2):121--140, 2001.

\bibitem{AHKC19}
K.~B. Ali, M.~Hsini, K.~Kefi, and N.~T. Chung.
\newblock On a nonlocal fractional {$p(\cdot,\cdot)$}-{L}aplacian problem with
  competing nonlinearities.
\newblock {\em Complex Anal. Oper. Theory}, 13(3):1377--1399, 2019.

\bibitem{Alk97}
Y.~A. Alkhutov.
\newblock The {H}arnack inequality and the {H}\"{o}lder property of solutions
  of nonlinear elliptic equations with a nonstandard growth condition.
\newblock {\em Differ. Uravn.}, 33(12):1651--1660, 1726, 1997.

\bibitem{ASSA21}
R.~Ayazoglu, Y.~Sara\c{c}, S.~\c{S}. \c{S}ener, and G.~Alisoy.
\newblock Existence and multiplicity of solutions for a
  {S}chr\"{o}dinger-{K}irchhoff type equation involving the fractional
  {$p(.,.)$}-{L}aplacian operator in {$\Bbb R^N$}.
\newblock {\em Collect. Math.}, 72(1):129--156, 2021.

\bibitem{ABS19}
E.~Azroul, A.~Benkirane, and M.~Shimi.
\newblock Eigenvalue problems involving the fractional {$p(x)$}-{L}aplacian
  operator.
\newblock {\em Adv. Oper. Theory}, 4(2):539--555, 2019.

\bibitem{ABS20}
E.~Azroul, A.~Benkirane, and M.~Shimi.
\newblock General fractional {S}obolev space with variable exponent and
  applications to nonlocal problems.
\newblock {\em Adv. Oper. Theory}, 5(4):1512--1540, 2020.

\bibitem{ABSS21}
E.~Azroul, A.~Benkirane, M.~Shimi, and M.~Srati.
\newblock On a class of fractional {$p(x)$}-{K}irchhoff type problems.
\newblock {\em Appl. Anal.}, 100(2):383--402, 2021.

\bibitem{BB19}
A.~Baalal and M.~Berghout.
\newblock Density properties for fractional {S}obolev spaces with variable
  exponents.
\newblock {\em Ann. Funct. Anal.}, 10(3):308--324, 2019.

\bibitem{Bah18}
A.~Bahrouni.
\newblock Comparison and sub-supersolution principles for the fractional
  {$p(x)$}-{L}aplacian.
\newblock {\em J. Math. Anal. Appl.}, 458(2):1363--1372, 2018.

\bibitem{BR18}
A.~Bahrouni and V.~D. R\u{a}dulescu.
\newblock On a new fractional {S}obolev space and applications to nonlocal
  variational problems with variable exponent.
\newblock {\em Discrete Contin. Dyn. Syst. Ser. S}, 11(3):379--389, 2018.

\bibitem{BRW20}
A.~Bahrouni, V.~D. R\u{a}dulescu, and P.~Winkert.
\newblock Robin fractional problems with symmetric variable growth.
\newblock {\em J. Math. Phys.}, 61(10):101503, 14, 2020.

\bibitem{BT21}
R.~Biswas and S.~Tiwari.
\newblock Variable order nonlocal {C}hoquard problem with variable exponents.
\newblock {\em Complex Var. Elliptic Equ.}, 66(5):853--875, 2021.

\bibitem{BA20}
A.~Boumazourh and E.~Azroul.
\newblock On a class of fractional systems with nonstandard growth conditions.
\newblock {\em J. Pseudo-Differ. Oper. Appl.}, 11(2):805--820, 2020.

\bibitem{BP16}
L.~Brasco and E.~Parini.
\newblock The second eigenvalue of the fractional {$p$}-{L}aplacian.
\newblock {\em Adv. Calc. Var.}, 9(4):323--355, 2016.

\bibitem{Bre12}
D.~Breit.
\newblock New regularity theorems for non-autonomous variational integrals with
  {$(p,q)$}-growth.
\newblock {\em Calc. Var. Partial Differential Equations}, 44(1-2):101--129,
  2012.

\bibitem{CCV11}
L.~Caffarelli, C.~H. Chan, and A.~Vasseur.
\newblock Regularity theory for parabolic nonlinear integral operators.
\newblock {\em J. Amer. Math. Soc.}, 24(3):849--869, 2011.

\bibitem{CGA20}
Y.~Cheng, B.~Ge, and R.~P. Agarwal.
\newblock Variable-order fractional {S}obolev spaces and nonlinear elliptic
  equations with variable exponents.
\newblock {\em J. Math. Phys.}, 61(7):071507, 12, 2020.

\bibitem{CPC97}
V.~Chiad\`o~Piat and A.~Coscia.
\newblock H\"{o}lder continuity of minimizers of functionals with variable
  growth exponent.
\newblock {\em Manuscripta Math.}, 93(3):283--299, 1997.

\bibitem{CM99}
A.~Coscia and G.~Mingione.
\newblock H\"{o}lder continuity of the gradient of {$p(x)$}-harmonic mappings.
\newblock {\em C. R. Acad. Sci. Paris S\'{e}r. I Math.}, 328(4):363--368, 1999.

\bibitem{Coz17}
M.~Cozzi.
\newblock Regularity results and {H}arnack inequalities for minimizers and
  solutions of nonlocal problems: a unified approach via fractional {D}e
  {G}iorgi classes.
\newblock {\em J. Funct. Anal.}, 272(11):4762--4837, 2017.

\bibitem{Coz19}
M.~Cozzi.
\newblock Fractional {D}e {G}iorgi classes and applications to nonlocal
  regularity theory.
\newblock In {\em Contemporary research in elliptic {PDE}s and related topics},
  volume~33 of {\em Springer INdAM Ser.}, pages 277--299. Springer, Cham, 2019.

\bibitem{CUF13}
D.~V. Cruz-Uribe and A.~Fiorenza.
\newblock {\em Variable {L}ebesgue spaces}.
\newblock Applied and Numerical Harmonic Analysis. Birkh\"{a}user/Springer,
  Heidelberg, 2013.
\newblock Foundations and harmonic analysis.

\bibitem{DPR17}
L.~M. Del~Pezzo and J.~D. Rossi.
\newblock Traces for fractional {S}obolev spaces with variable exponents.
\newblock {\em Adv. Oper. Theory}, 2(4):435--446, 2017.

\bibitem{DCKP14}
A.~Di~Castro, T.~Kuusi, and G.~Palatucci.
\newblock Nonlocal {H}arnack inequalities.
\newblock {\em J. Funct. Anal.}, 267(6):1807--1836, 2014.

\bibitem{DCKP16}
A.~Di~Castro, T.~Kuusi, and G.~Palatucci.
\newblock Local behavior of fractional {$p$}-minimizers.
\newblock {\em Ann. Inst. H. Poincar\'{e} Anal. Non Lin\'{e}aire},
  33(5):1279--1299, 2016.

\bibitem{DNPV12}
E.~Di~Nezza, G.~Palatucci, and E.~Valdinoci.
\newblock Hitchhiker's guide to the fractional {S}obolev spaces.
\newblock {\em Bull. Sci. Math.}, 136(5):521--573, 2012.

\bibitem{Die04}
L.~Diening.
\newblock Maximal function on generalized {L}ebesgue spaces {$L^{p(\cdot)}$}.
\newblock {\em Math. Inequal. Appl.}, 7(2):245--253, 2004.

\bibitem{DHHR11}
L.~Diening, P.~Harjulehto, P.~H\"{a}st\"{o}, and M.~R{\r u}\v{z}i\v{c}ka.
\newblock {\em Lebesgue and {S}obolev spaces with variable exponents}, volume
  2017 of {\em Lecture Notes in Mathematics}.
\newblock Springer, Heidelberg, 2011.

\bibitem{DHR09}
L.~Diening, P.~H\"{a}st\"{o}, and S.~Roudenko.
\newblock Function spaces of variable smoothness and integrability.
\newblock {\em J. Funct. Anal.}, 256(6):1731--1768, 2009.

\bibitem{DZZ20}
M.~Ding, C.~Zhang, and S.~Zhou.
\newblock On optimal {$C^{1,\alpha}$} estimates for {$p(x)$}-{L}aplace type
  equations.
\newblock {\em Nonlinear Anal.}, 200:112030, 14, 2020.

\bibitem{DK20}
B.~Dyda and M.~Kassmann.
\newblock Regularity estimates for elliptic nonlocal operators.
\newblock {\em Anal. PDE}, 13(2):317--370, 2020.

\bibitem{FZ99}
X.~Fan and D.~Zhao.
\newblock A class of {D}e {G}iorgi type and {H}\"{o}lder continuity.
\newblock {\em Nonlinear Anal.}, 36(3, Ser. A: Theory Methods):295--318, 1999.

\bibitem{FZ01}
X.~Fan and D.~Zhao.
\newblock On the spaces {$L^{p(x)}(\Omega)$} and {$W^{m,p(x)}(\Omega)$}.
\newblock {\em J. Math. Anal. Appl.}, 263(2):424--446, 2001.

\bibitem{FK13}
M.~Felsinger and M.~Kassmann.
\newblock Local regularity for parabolic nonlocal operators.
\newblock {\em Comm. Partial Differential Equations}, 38(9):1539--1573, 2013.

\bibitem{HHLN10}
P.~Harjulehto, P.~H\"{a}st\"{o}, U.~V. L\^{e}, and M.~Nuortio.
\newblock Overview of differential equations with non-standard growth.
\newblock {\em Nonlinear Anal.}, 72(12):4551--4574, 2010.

\bibitem{HKLMP08}
P.~Harjulehto, T.~Kuusi, T.~Lukkari, N.~Marola, and M.~Parviainen.
\newblock Harnack's inequality for quasiminimizers with nonstandard growth
  conditions.
\newblock {\em J. Math. Anal. Appl.}, 344(1):504--520, 2008.

\bibitem{HK19}
K.~Ho and Y.-H. Kim.
\newblock A-priori bounds and multiplicity of solutions for nonlinear elliptic
  problems involving the fractional {$p(\cdot)$}-{L}aplacian.
\newblock {\em Nonlinear Anal.}, 188:179--201, 2019.

\bibitem{HK21}
K.~Ho and Y.-H. Kim.
\newblock The concentration-compactness principles for
  {$W^{s,p(\cdot,\cdot)}(\Bbb R^N)$} and application.
\newblock {\em Adv. Nonlinear Anal.}, 10(1):816--848, 2021.

\bibitem{Kas09}
M.~Kassmann.
\newblock A priori estimates for integro-differential operators with measurable
  kernels.
\newblock {\em Calc. Var. Partial Differential Equations}, 34(1):1--21, 2009.

\bibitem{KS14}
M.~Kassmann and R.~W. Schwab.
\newblock Regularity results for nonlocal parabolic equations.
\newblock {\em Riv. Math. Univ. Parma (N.S.)}, 5(1):183--212, 2014.

\bibitem{KRV17}
U.~Kaufmann, J.~D. Rossi, and R.~Vidal.
\newblock Fractional {S}obolev spaces with variable exponents and fractional
  {$p(x)$}-{L}aplacians.
\newblock {\em Electron. J. Qual. Theory Differ. Equ.}, pages Paper No. 76, 10,
  2017.

\bibitem{KR91}
O.~Kov\'{a}\v{c}ik and J.~R\'{a}kosn\'{\i}k.
\newblock On spaces {$L^{p(x)}$} and {$W^{k,p(x)}$}.
\newblock {\em Czechoslovak Math. J.}, 41(116)(4):592--618, 1991.

\bibitem{KMS15}
T.~Kuusi, G.~Mingione, and Y.~Sire.
\newblock Nonlocal self-improving properties.
\newblock {\em Anal. PDE}, 8(1):57--114, 2015.

\bibitem{Mar89}
P.~Marcellini.
\newblock Regularity of minimizers of integrals of the calculus of variations
  with nonstandard growth conditions.
\newblock {\em Arch. Rational Mech. Anal.}, 105(3):267--284, 1989.

\bibitem{Mar91}
P.~Marcellini.
\newblock Regularity and existence of solutions of elliptic equations with
  {$p,q$}-growth conditions.
\newblock {\em J. Differential Equations}, 90(1):1--30, 1991.

\bibitem{Ok18}
J.~Ok.
\newblock Harnack inequality for a class of functionals with non-standard
  growth via {D}e {G}iorgi's method.
\newblock {\em Adv. Nonlinear Anal.}, 7(2):167--182, 2018.

\bibitem{Ok21}
J.~Ok.
\newblock Local {H}\"older regularity for nonlocal equations with
  variable powers.
\newblock {\em arXiv preprint arXiv:2107.06611}, 2021.

\bibitem{ZZ20}
C.~Zhang and X.~Zhang.
\newblock Renormalized solutions for the fractional {$p(x)$}-{L}aplacian
  equation with {$L^1$} data.
\newblock {\em Nonlinear Anal.}, 190:111610, 15, 2020.

\bibitem{Zhi93}
V.~Zhikov.
\newblock Lavrentiev phenomenon and homogenization for some variational
  problems.
\newblock {\em C. R. Acad. Sci. Paris S\'{e}r. I Math.}, 316(5):435--439, 1993.

\bibitem{Zhi86}
V.~V. Zhikov.
\newblock Averaging of functionals of the calculus of variations and elasticity
  theory.
\newblock {\em Izv. Akad. Nauk SSSR Ser. Mat.}, 50(4):675--710, 877, 1986.

\bibitem{Zhi95}
V.~V. Zhikov.
\newblock On {L}avrentiev's phenomenon.
\newblock {\em Russian J. Math. Phys.}, 3(2):249--269, 1995.

\bibitem{Zhi97}
V.~V. Zhikov.
\newblock On some variational problems.
\newblock {\em Russian J. Math. Phys.}, 5(1):105--116 (1998), 1997.

\bibitem{ZAF21}
J.~Zuo, T.~An, and A.~Fiscella.
\newblock A critical {K}irchhoff-type problem driven by a
  {$p(\cdot)$}-fractional {L}aplace operator with variable {$s(\cdot)$}-order.
\newblock {\em Math. Methods Appl. Sci.}, 44(1):1071--1085, 2021.

\end{thebibliography}

\end{document}